\documentclass[11pt]{amsart}
\usepackage{amsmath,amssymb,latexsym,amsmath,amsthm,amscd}
\usepackage[left=3cm,top=4cm,right=3cm,bottom = 4cm]{geometry}

\theoremstyle{plain}
\newtheorem{thm}{Theorem}
\newtheorem{lem}[thm]{Lemma}
\newtheorem{prop}[thm]{Proposition}

\theoremstyle{definition}
\newtheorem{dfn}[thm]{Definition}
\newtheorem{ex}[thm]{Example}

\theoremstyle{remark}
\newtheorem{rmk}[thm]{Remark}

\DeclareMathOperator{\End}{End}

\DeclareMathOperator{\Ind}{Ind}
\DeclareMathOperator{\Res}{Res}

\begin{document}

\title{On Indecomposable Vertex Algebras associated with Vertex Algebroids }
\author{Phichet Jitjankarn, and Gaywalee Yamskulna}
\address{School of Science, Walailak University, Nakhon Si Thammarat, Thailand}
\email{jitjankarn@gmail.com}
\address{Department of Mathematics\\ Illinois State University, Normal, IL, USA}
\email{gyamsku@ilstu.edu}


\begin{abstract} Let $A$ be a finite dimensional unital commutative associative algebra and let $B$ be a finite dimensional vertex $A$-algebroid such that its Levi factor is isomorphic to $sl_2$. Under suitable conditions, we construct an indecomposable non-simple $\mathbb{N}$-graded vertex algebra $\overline{V_B}$ from the $\mathbb{N}$-graded vertex algebra $V_B$ associated with the vertex $A$-algebroid $B$. We show that this indecomposable non-simple $\mathbb{N}$-graded vertex algebra $\overline{V_B}$ is $C_2$-cofinite and has only two irreducible modules.\end{abstract}
\maketitle
\section{Introduction} 

\noindent It is well known that the $C_2$-cofiniteness property of vertex (operator) algebras plays an important role in the study of representation theory of vertex (operator) algebras (e.g. \cite{A2}, \cite{ ABD}, \cite{Bu}, \cite{DLM}, \cite{GN}, \cite{Mi1}-\cite{Mi3}, \cite{ Z}). Over the years, rational $C_2$-cofinite vertex algebras have been studied intensively (e.g. \cite{B1}-\cite{B2}, \cite{D}, \cite{DoL}, \cite{DLM}, \cite{DoM1}, \cite{FLM2}, \cite{FZ}) . However, the literature devoted to the study of irrational $C_2$-cofinite vertex algebras is sparse. In fact,  there are very few known examples of families of irrational $C_2$-cofinite vertex (super)algebras (e.g. \cite{A}, \cite{AdM1}-\cite{AdM4}, \cite{CF}, \cite{FFHST}, \cite{FGST1}-\cite{FGST3}). Therefore, it is crucial to study these known examples and to seek for new models.

\vspace{0.2cm}

\noindent The aim of this paper is to construct indecomposable non-simple vertex algebras that satisfy the $C_2$-condition from $\mathbb{N}$-graded vertex algebras associated with vertex algebroids. For a $\mathbb{N}$-graded vertex algebra $V=\oplus_{n=0}^{\infty}V_{(n)}$ such that $\dim V_{(0)}\geq 2$, it is known widely that $V_{(0)}$ is a unital commutative associative algebra and $V_{(1)}$ is a vertex $V_{(0)}$-algebroid. In \cite{GMS}, among other important things, Gorbounov, Malikov and Schechtman constructed a $\mathbb{N}$-graded vertex algebra $V=\oplus_{n=0}^{\infty}V_{(n)}$ from any vertex $A$-algebroid, such that $V_{(0)}=A$ and the vertex $A$-algebroid $V_{(1)}$ is isomorphic to the given vertex $A$-algebroid. The classification of graded simple non-twisted and twisted modules for the vertex algebras associated with vertex algebroids had been studied in \cite{LiY}-\cite{LiY2} by Li and the second author of this paper. 

\vspace{0.2cm}

\noindent In terms of general theory of $\mathbb{N}$-graded vertex algebras, Dong and Mason showed that a $\mathbb{N}$-graded vertex operator algebra $V$ is local if and only if $V_{(0)}$ is a local algebra. Moreover, indecomposibility of $V$ is equivalent to $V_{(0)}$ being a local algebra \cite{DoM}. Note that in order to prove this statement one needs to have a Virasoro element. In \cite{JY}, we explored criteria for $\mathbb{N}$-graded vertex algebras $V=\oplus_{n=0}^{\infty}V_{(n)}$ such that $\dim V_{(0)}\geq 2$ to be indecomposable non-simple vertex algebras and studied influences of semisimple Leibniz algebras on the algebraic structure of this type of vertex algebras. Precisely, we provided tools to characterize indecomposable non-simple $\mathbb{N}$-graded vertex algebras. Also, we examined the algebraic structure of $\mathbb{N}$-graded vertex algebras $V=\oplus_{n=0}^{\infty}V_{(n)}$ that is generated by $V_{(0)}$ and $V_{(1)}$ such that $\dim~V_{(0)}\geq 2$ and $V_{(1)}$ is a (semi)simple Leibniz algebra that has $sl_2$ as its Levi factor. We showed that under suitable conditions this type of vertex algebra is indecomposable non-simple. 

\vspace{0.2cm}

\noindent In this paper, we continue our investigation on $\mathbb{N}$-graded indecomposable non-simple vertex algebras. First, we apply results in \cite{JY} to show that for a given finite dimensional vertex $A$-algebroid $B$ such that $B$ is a (semi)simple Leibniz algebra that has $sl_2$ as its Levi factor, the $\mathbb{N}$-graded vertex algebra associated with the vertex $A$-algebroid $B$ is indecomposable and non-simple. Moreover, we establish the following results.

\begin{thm}\label{main1} Let $A$ be a finite-dimensional commutative associative algebra with the identity $\mathfrak{e}$ such that $dim~A\geq 2$. Let $B$ be a finite-dimensional vertex $A$-algebroid such that $A$ is not a trivial $B$-module and $Leib(B)\neq \{0\}$. Let $S$ be the Levi factor of the Leibniz algebra $B$ such that $S=Span\{e,f,h\}$, $e_0f=h$, $h_0e=2e$, $h_0f=-2f$, and $e_1f=k\mathfrak{e}$. Here, $k\in\mathbb{C}\backslash\{0\}$. Assume that one of the following statements hold.

\vspace{0.2cm}

\noindent (I) $B$ is simple Leibniz algebra;

\vspace{0.2cm}

\noindent (II) $B$ is a semisimple Leibniz algebra and $Ker(\partial)=\{a\in A~|~b_0a=0\text{ for all }b\in B\}$. 

\vspace{0.2cm}

\noindent We then have the following results:
\vspace{0.2cm}

\noindent (i) the $\mathbb{N}$-graded vertex algebra $V_B(=\oplus_{n=0}^{\infty}(V_B)_{(n)})$ associated with the vertex $A$-algebroid $B$ is indecomposable non-simple. 

\vspace{0.2cm}

\noindent (ii) The set of representatives of equivalence classes of  finite-dimensional simple $sl_2$-modules is equivalent to the set of representatives of equivalence classes of $\mathbb{N}$-graded simple $V_B$-modules $N=\oplus_{n=0}^{\infty}N_{(n)}$ such that $dim~N_{(0)}<\infty$.
\end{thm}

\vspace{0.2cm}

\noindent Next, we show that  a certain quotient space of $V_B$, constructed in Theorem \ref{main1}, is an indecomposable non-simple vertex algebra that satisfies the $C_2$-condition and has only two irreducible modules. Precisely, we establish the following results.

\begin{thm}\label{main2} Let $A$ be a finite-dimensional commutative associative algebra with the identity $\mathfrak{e}$ such that $dim~A\geq 2$. Let $B$ be a finite-dimensional vertex $A$-algebroid that satisfies the given conditions in Theorem \ref{main1}. Let $(e(-1)e)$ be an ideal of $V_B$ that is generated by $e(-1)e$. Then 

\vspace{0.2cm}

\noindent (i) $(e(-1)e)\cap A=\{0\}$, $(e(-1)e)\cap B=\{0\}$. Moreover, the $\mathbb{N}$-graded vertex algebra $V_B/(e(-1)e)$ is indecomposable non-simple.

\vspace{0.2cm}

\noindent (ii) The vertex algebra $V_B/(e(-1)e)$ satisfies the $C_2$-condition.

\vspace{0.2cm}

\noindent (iii) Let $L=\mathbb{Z}\alpha$ be a rank one positive definite even lattice equipped with a $\mathbb{Q}$-valued $\mathbb{Z}$-bilinear form $(\cdot,\cdot)$ such that $(\alpha,\alpha)=2$. Then $V_L$ and $V_{L+\frac{1}{2}\alpha}$ are the only two irreducible $V_B/(e(-1)e)$-modules. 
\end{thm}

\vspace{0.2cm}

\noindent This paper is organized as follows: in Section 2, we first review properties of Leibniz algebras, 1-truncated conformal algebras, and vertex algebroids. We discuss about vertex algebroids associated with (semi) simple Leibniz algebras that have $sl_2$ as their Levi factor. Also, we give necessary background on vertex algebras and recall construction of vertex algebras associated with vertex algebroids, and their graded simple modules. In Section 3, we give the proof of Theorem \ref{main1}, and Theorem \ref{main2}. We include some appendices in Section 4 containing background on a vertex operator algebra associated with a certain type of rank one positive definite even lattices, and vertex operator algebras associated with highest weight representations of affine Lie algebras. 

\section{Preliminaries}
\subsection{Leibniz Algebras}
\begin{dfn}(\cite{DMS}, \cite{FM})\ \ 

\vspace{0.2cm}

\noindent (i) A {\em left Leibniz algebra} $\mathfrak{L}$ is a $\mathbb{C}$-vector space equipped with a bilinear map $[~,~]:\mathfrak{L}\times\mathfrak{L}\rightarrow\mathfrak{L}$ satisfying the Leibniz identity $[a,[b,c]]=[[a,b],c]+[b,[a,c]]$ for all $a,b,c\in\mathfrak{L}$.

\vspace{0.2cm}

\noindent (ii) Let $\mathfrak{L}$ be a left Leibniz algebra over $\mathbb{C}$. Let $I$ be a subspace of $\mathfrak{L}$. $I$ is a {\em left} (respectively, {\em right}) {\em ideal} of $\mathfrak{L}$ if $[\mathfrak{L}, I]\subseteq I$ (respectively, $[I,\mathfrak{L}]\subseteq I$). $I$ is an {\em ideal} of $\mathfrak{L}$ if it is both a left and a right ideal. 
\end{dfn}

\begin{ex} We define $Leib(\mathfrak{L})=Span\{~[u,u]~|~u\in\mathfrak{L}~\}=Span\{[u,v]+[v,u]~|~u,v\in\mathfrak{L}\}$. $Leib(\mathfrak{L})$ is an ideal of $\mathfrak{L}$. Moreover, for $v,w\in Leib(\mathfrak{L})$, $[v,w]=0$. 
\end{ex}

\begin{dfn}\cite{DMS} Let $(\mathfrak{L}, [~,~])$ be a left Leibniz algebra. The series of ideals $$...\subseteq \mathfrak{L}^{(2)}\subseteq \mathfrak{L}^{(1)}\subseteq \mathfrak{L}$$ where $\mathfrak{L}^{(1)}=[\mathfrak{L},\mathfrak{L}]$, $\mathfrak{L}^{(i+1)}=[\mathfrak{L}^{(i)},\mathfrak{L}^{(i)}]$ is called the {\em derived series} of $\mathfrak{L}$.  A left Leibniz algebra $\mathfrak{L}$ is {\em solvable} if $\mathfrak{L}^{(m)}=0$ for some integer $m\geq 0$. As in the case of Lie algebras, any left Leibniz algebra $\mathfrak{L}$ contains a unique maximal solvable ideal $rad(\mathfrak{L})$ called the  the {\em radical} of $\mathfrak{L}$ which contains all solvable ideals.
\end{dfn}

\begin{ex} $Leib(\mathfrak{L})$ is a solvable ideal.
\end{ex}

\begin{dfn}\cite{DMS}\ \  

\vspace{0.2cm}

\noindent (i) A left Leibniz algebra $\mathfrak{L}$ is {\em simple}  if $[\mathfrak{L},\mathfrak{L}]\neq Leib(\mathfrak{L})$, and $\{0\}$, $Leib(\mathfrak{L})$, $\mathfrak{L}$ are the only ideals of $\mathfrak{L}$.

\vspace{0.2cm}

\noindent (ii) A left Leibniz algebra $\mathfrak{L}$ is said to be {\em semisimple} if $rad(\mathfrak{L})=Leib(\mathfrak{L})$. 
\end{dfn}
\begin{prop}(\cite{Ba2}, \cite{DMS}) Let $\mathfrak{L}$ be a left Leibniz algebra. 

\vspace{0.2cm}

\noindent (i) There exists a subalgebra $S$ which is a semisimple Lie algebra of $\mathfrak{L}$ such that $\mathfrak{L}=S \dot{+} rad(\mathfrak{L})$. As in the case of a Lie algebra, we call $S$ a Levi subalgebra or a Levi factor of $\mathfrak{L}$.

\vspace{0.2cm}

\noindent (ii) If $\mathfrak{L}$ is a semisimple Leibniz algebra then $\mathfrak{L}=(S_1\oplus S_2\oplus...\oplus S_k)\dot{+}Leib(\mathfrak{L})$, where $S_j$ is a simple Lie algebra for all $1\leq j\leq k$. Moreover, $[\mathfrak{L},\mathfrak{L}]=\mathfrak{L}$. 

\vspace{0.2cm}

\noindent (iii) If $\mathfrak{L}$ is a simple Leibniz algebra, then there exists a simple Lie algebra $S$ such that $Leib(\mathfrak{L})$ is an irreducible module over $S$ and $\mathfrak{L}=S\dot{+}Leib(\mathfrak{L})$.
\end{prop}

\begin{dfn} Let $\mathfrak{L}$ be a left Leibniz algebra. A left $\mathfrak{L}$-module is a vector space $M$ equipped with a $\mathbb{C}$-bilinear map $\mathfrak{L}\times M\rightarrow M; (u,m)\mapsto u\cdot m$ such that $([u,v])\cdot m=u\cdot (v\cdot m)-v\cdot(u\cdot m)$ for all $u,v\in \mathfrak{L},m\in M$. 

\vspace{0.2cm}

\noindent The usual definitions of the notions of submodule, irreducibility, complete reducibility, homomorphism, isomorphism, etc., hold for left Leibniz modules.
\end{dfn}

\begin{rmk} $Leib(\mathfrak{L})$ acts as zero on $M$. 
\end{rmk}

\subsection{1-Truncated Conformal Algebras and Vertex Algebroids}

\begin{dfn}\cite{GMS} A {\em 1-truncated conformal algebra} is a graded vector space $C=C_0\oplus C_1$ equipped with a linear map $\partial:C_0\rightarrow C_1$ and bilinear operations $(u,v)\mapsto u_iv$ for $i=0,1$ of degree $-i-1$ on $C=C_0\oplus C_1$ such that the following axioms hold:

\medskip

\noindent(Derivation) for $a\in C_0$, $u\in C_1$,
\begin{equation}
(\partial a)_0=0,\ \ (\partial a)_1=-a_0,\ \ \partial(u_0a)=u_0\partial a;
\end{equation}

\noindent(Commutativity) for $a\in C_0$, $u,v\in C_1$,
\begin{equation} 
u_0a=-a_0u,\ \ u_0v=-v_0u+\partial(u_1v),\ \ u_1v=v_1u;
\end{equation}

\noindent(Associativity) for $\alpha,\beta,\gamma\in C_0\oplus C_1$,
\begin{equation}
\alpha_0\beta_i\gamma=\beta_i\alpha_0\gamma+(\alpha_0\beta)_i\gamma.
\end{equation}
\end{dfn}

\begin{dfn}(\cite{Br1}, \cite{Br2}, \cite{GMS}) Let $(A,*)$ be a unital commutative associative algebra over $\mathbb{C}$ with the identity $1$. A {\em vertex $A$-algebroid} is a $\mathbb{C}$-vector space $\Gamma$ equipped with 
\begin{enumerate}
\item a $\mathbb{C}$-bilinear map $A\times \Gamma\rightarrow \Gamma, \ \ (a,v)\mapsto a\cdot v$ such that $1\cdot v=v$ (i.e. a nonassociative unital $A$-module),
\item a structure of a Leibniz $\mathbb{C}$-algebra $[~,~]:\Gamma\times \Gamma\rightarrow\Gamma$, 
\item a homomorphism of Leibniz $\mathbb{C}$-algebra $\pi:\Gamma\rightarrow Der(A)$,
\item a symmetric $\mathbb{C}$-bilinear pairing $\langle ~,~\rangle:\Gamma\otimes_{\mathbb{C}}\Gamma\rightarrow A$,
\item a $\mathbb{C}$-linear map $\partial :A\rightarrow \Gamma$ such that $\pi\circ \partial =0$ which satisfying the following conditions:
\begin{eqnarray*}
&&a\cdot (a'\cdot v)-(a*a')\cdot v=\pi(v)(a)\cdot \partial(a')+\pi(v)(a')\cdot \partial(a),\\
&&[u,a\cdot v]=\pi(u)(a)\cdot v+a\cdot [u,v],\\
&&[u,v]+[v,u]=\partial(\langle u,v\rangle),\\
&&\pi(a\cdot v)=a\pi(v),\\
&&\langle a\cdot u,v\rangle=a*\langle u,v\rangle-\pi(u)(\pi(v)(a)),\\
&&\pi(v)(\langle v_1,v_2\rangle)=\langle [v,v_1],v_2\rangle+\langle v_1,[v,v_2]\rangle,\\
&&\partial(a*a')=a\cdot \partial(a')+a'\cdot\partial(a),\\
&&[v,\partial(a)]=\partial(\pi(v)(a)),\\
&&\langle v,\partial(a)\rangle=\pi(v)(a)
\end{eqnarray*}
for $a,a'\in A$, $u,v,v_1,v_2\in\Gamma$.
\end{enumerate}
\end{dfn}

\begin{prop}\cite{LiY} Let $(A,*)$ be a unital commutative associative algebra and let $B$ be a module for $A$ as a nonassociative algebra . Then a vertex $A$-algebroid structure on $B$ exactly amounts to a 1-truncated conformal algebra structure on $C=A\oplus B$ with 
\begin{eqnarray*}
&&a_ia'=0,\\
&&u_0v=[u,v],~u_1v=\langle u,v\rangle,\\
&&u_0a=\pi(u)(a),~ a_0u=-u_0a
\end{eqnarray*} for $a,a'\in A$, $u,v\in B$, $i=0,1$ such that 
\begin{eqnarray*}
&&a\cdot(a'\cdot u)-(a*a')\cdot u=(u_0a)\cdot \partial a'+(u_0a')\cdot \partial a,\\
&&u_0(a\cdot v)-a\cdot (u_0v)=(u_0a)\cdot v,\\
&&u_0(a*a')=a*(u_0a')+(u_0a)*a',\\
&&a_0(a'\cdot v)=a'*(a_0v),\\
&&(a\cdot u)_1v=a*(u_1v)-u_0v_0a,\\
&&\partial(a*a')=a\cdot \partial(a')+a'\cdot \partial(a).
\end{eqnarray*}
\end{prop}
\vspace{0.2cm}

\noindent For the rest of this section, we assume the following:

\noindent (i) $(A,*)$ is a unital commutative associative algebra with the identity $\mathfrak{e}$ and $dim (A)<\infty$. 

\noindent (ii) $B$ is a vertex $A$-algebroid such that $dim (B)<\infty$, and $A$ is not a trivial $B$-module.

\vspace{0.2cm}

\noindent Recall that a set $I$ is called an ideal of a vertex $A$-algebroid $B$ if $I$ is a left ideal of Leibniz algebra $B$ and $a\cdot u\in I$ for all $a\in A, u\in I$. 
\begin{ex} We set $A\partial(A)=Span\{a\cdot \partial(a')~|~a,a'\in A\}$. $A\partial(A)$ is an ideal of a vertex $A$-algebroid $B$. In fact, $A\partial(A)$ is an abelian Lie algebra. 
\end{ex}

\begin{prop}\label{Bsimple}\cite{JY} Let $B$ be a simple Leibniz algebra such that $Leib(B)\neq\{0\}$. 
Assume that its Levi factor $S=Span\{e,f,h\}$ such that $e_0f=h$, $h_0e=2e$, $h_0f=-2f$, and $e_1f=k\mathfrak{e}\in(\mathbb{C}\mathfrak{e})\backslash\{0\}$. Then 

\vspace{0.3cm}

\noindent (i) $e_1e=f_1f=e_1h=f_1h=0$, $k=1$, $h_1h=2\mathfrak{e}$.

\vspace{0.2cm}

\noindent (ii) $Ker(\partial)=\mathbb{C}\mathfrak{e}$

\vspace{0.2cm}

\noindent (iii) $Leib(B)$ is an irreducible $sl_2$-module of dimension 2. Moreover, as a $sl_2$-module, $A$ is a direct sum of a trivial module and an irreducible $sl_2$-module of dimension 2. 

\vspace{0.2cm}

\noindent (iv) $A$ is a local algebra. Let $A_{\neq 0}$ be an irreducible $sl_2$-submodule of $A$ that has dimension 2. Let $a_0$ be the highest weight vector of $A_{\neq 0}$ of weight 1 and let $a_1=f_0a_0$. Hence, the set $\{a_0,a_1\}$ forms a basis of $A_{\neq 0}$, the set $\{\mathfrak{e},a_0,a_1\}$ is a basis of $A$, and the set $\{\partial(a_0),\partial(a_1)\}$ is a basis of $Leib(B)$. 

\vspace{0.2cm}

\noindent Relationships among $a_0,a_1,e,f,h,\partial(a_0),\partial(a_1)$ are desribed below:
\begin{eqnarray}
&&(\partial(a_0))_1e=0,~(\partial(a_0))_1f=a_1,~(\partial(a_0))_1h=a_0,\\
&&(\partial(a_1))_1e=a_0,~(\partial(a_1))_1f=0,~(\partial(a_1))_1h=-a_1,\\
&&a_0\cdot e=0,~a_0\cdot f=\partial(a_1),~a_0\cdot h=\partial(a_0),~a_0\cdot \partial(a_i)=0\text{ for }i\in\{0,1\},\\
&&a_1\cdot e=\partial(a_0),~a_1\cdot f=0,~a_1\cdot h=-\partial(a_1),~a_1\cdot \partial(a_i)=0\text{ for }i\in\{0,1\},\\
&&a_i*a_j=0\text{ for all }i,j\in\{0,1\}.
\end{eqnarray}
\end{prop}

\begin{prop}\label{Bsemisimple}\cite{JY} Let $B$ be a semisimple Leibniz algebra such that $Leib(B)\neq\{0\}$, and $Ker(\partial)=\{a\in A~|~u_0a=0\text{ for all }u\in B\}$. 

\vspace{0.2cm}

\noindent Assume that the Levi factor $S=Span\{e,f,h\}$ such that $e_0f=h, h_0e=2e, h_0f=-2f$ and $e_1f=k\mathfrak{e}\in\mathbb{C}\mathfrak{e}\backslash \{0\}$. We set $A=\mathbb{C}\mathfrak{e}\oplus_{j=1}^l N^j$ where each $N^j$ is an irreducible $sl_2$-submodule of $A$. Then
 
\vspace{0.2cm}

\noindent (i) $e_1e=f_1f=e_1h=f_1h=0$, $k=1$, $h_1h=2\mathfrak{e}$;

\vspace{0.2cm}

\noindent (ii) $Ker(\partial)=\mathbb{C}\mathfrak{e}$;

\vspace{0.2cm}

\noindent (iii) For $j\in\{1,...,l\}$ $\dim N^j=2$, and $\dim Leib(B)=2l$; 

\vspace{0.2cm}

\noindent (iv) $A$ is a local algebra. For each $j$, we let $a_{j,0}$ be a highest weight vector of $N^j$ and $a_{j,1}=f_0(a_{j,0})$. Then $\{\mathfrak{e}, a_{j,i}~|~j\in \{1,....,l\},~i\in\{0,1\}\}$ is a basis of $A$, and $\{\partial(a_{j,i})~|~j\in\{1,...,l\},~i\in\{0,1\}\}$ is a basis of $Leib(B)$.  

\vspace{0.2cm}

\noindent Relations among $a_{j,i},e,f,h,\partial(a_{j,i})$ are described below: 
\begin{eqnarray}
&&a_{j,i}*a_{j',i'}=0,\label{rel1}\\
&&a_{j,0}\cdot e=0,~a_{j,1}\cdot e=\partial(a_{j,0}),\\
&&a_{j,0}\cdot f=\partial(a_{j,1}),~a_{j,1}\cdot f=0,\\
&&a_{j,0}\cdot h=\partial(a_{j,0}),~a_{j,1}\cdot h=-\partial(a_{j,1}),\\
&&a_{j,i}\cdot \partial(a_{j',i'})=0,\\
&&\partial(a_{j,i})_1e=e_0a_{j,i}=(2-i)a_{j,i-1},\\
&&\partial(a_{j,i})_1f=f_0a_{j,i}=(i+1)a_{j,i+1},\\
&&\partial(a_{j,i})_1h=h_0a_{j,i}=(1-2i)a_{j,i}.
\end{eqnarray}
\end{prop}

\subsection{Vertex Algebras} 

\ \

\begin{dfn}(\cite{B1}, \cite{FLM2}, \cite{LLi}) A {\em vertex algebra} is a vector space $V$ equipped with a linear map 
\begin{eqnarray*}
Y:V&&\rightarrow \End(V)[[x,x^{-1}]]\\
v&&\mapsto Y(v,x)=\sum_{n\in\mathbb{Z}}v_nx^{-n-1}\text{ where } v_n\in\End(V)
\end{eqnarray*} 
and equipped with a distinguished vector ${\bf 1}$, the \em{vacuum vector}, such that for $u,v\in V$, 
\begin{eqnarray*}
&&u_nv=0\text{ for $n$ sufficiently large},\\
&&Y({\bf 1},x)=1,\\
&&Y(v,x){\bf 1}\in V[[x]],\text{ and }\lim_{x\rightarrow 0}Y(v,x){\bf 1}=v
\end{eqnarray*}
and such that 
\begin{eqnarray*}
&&x_0^{-1}\delta\left(\frac{x_1-x_2}{x_0}\right)Y(u,x_1)Y(v,x_2)-x_0^{-1}\delta\left(\frac{x_2-x_1}{-x_0}\right)Y(v,x_2)Y(u,x_1)\\
&&=x_2^{-1}\delta\left(\frac{x_1-x_0}{x_2}\right)Y(Y(u,x_0)v,x_2)
\end{eqnarray*}
the {\em Jacobi identity}.
\end{dfn}

\vspace{0.2cm}

\noindent From the Jacobi identity we have Borcherds' commutator formula and iterate formula:
\begin{eqnarray}
&&{[u_m,v_n]}=\sum_{i\geq 0}{m\choose i}(u_iv)_{m+n-i}\\
&&(u_mv)_nw=\sum_{i\geq 0}(-1)^i{m\choose i}(u_{m-i}v_{n+i}w-(-1)^mv_{m+n-i}u_iw)
\end{eqnarray}
for $u,v,w\in V$, $m,n\in \mathbb{Z}$. 

\vspace{0.2cm}

\noindent We define a linear operator $D$ on $V$ by $D(v)=v_{-2}{\bf 1}$ for $v\in V$. Then 
\begin{eqnarray*}
&&Y(v,x){\bf 1}=e^{xD}v\text{ for }v\in V,\text{ and }\\
&&[D,v_n]=(Dv)_n=-nv_{n-1}\text{ for }v\in V, ~n\in\mathbb{Z}.
\end{eqnarray*} Moreover, for $u,v\in V$, we have $Y(u,x)v =e^{xD}Y(v,-x)u$ (skew-symmetry).

\vspace{0.2cm} 

\noindent A vertex algebra $V$ equipped with a $\mathbb{Z}$-grading $V=\oplus_{n\in \mathbb{Z}}V_{(n)}$ is called a {\em $\mathbb{Z}$-graded vertex algebra} if ${\bf 1}\in V_{(0)}$ and if $u\in V_{(k)}$ with $k\in\mathbb{Z}$ and for $m,n\in\mathbb{Z}$, $u_mV_{(n)}\subseteq V_{(k+n-m-1)}$. 

\vspace{0.2cm}

\noindent A $\mathbb{N}$-graded vertex algebra is defined in the obvious way.

\begin{prop}\cite{GMS} 

\vspace{0.2cm}

\noindent If $V=\oplus_{n\in \mathbb{N}}V_{(n)}$ is an $\mathbb{N}$-graded vertex algebra then 

\vspace{0.2cm}

\noindent (i) $V_{(0)}$ is a commutative associative algebra with identity ${\bf 1}$ and $V_{(1)}$ is Leibniz algebra.

\noindent (ii) In fact, $V_{(0)}\oplus V_{(1)}$ is a 1-truncated conformal algebra.

\noindent (iii) Moreover, $V_{(1)}$ is a vertex $V_{(0)}$-algebroid.
\end{prop}

\begin{prop}\label{Vindecomposable}\cite{JY} Let $V=\oplus_{n=0}^{\infty}V_{(n)}$ be a $\mathbb{N}$-graded vertex algebra that satisfies the following properties:

\vspace{0.2cm}

\noindent (a) $2\leq \dim V_{(0)}<\infty$, $1\leq dim V_{(1)}<\infty$, $V$ is generated by $V_{(0)}$ and $V_{(1)}$;

\vspace{0.2cm}

\noindent (b) $V_{(0)}$ is not a trivial module of a Leibniz algebra $V_{(1)}$, $u_0u\neq 0$ for some $u\in V_{(1)}$;

\vspace{0.2cm}

\noindent (c) the Levi factor of $V_{(1)}$ equals $Span\{e,f,h\}$, $e_0f=h$, $h_0e=2e$, $h_0f=-2f$ and $e_1f=k{\bf 1}$. Here, $k\in\mathbb{C}\backslash \{0\}$.

\vspace{0.2cm}

\noindent Assume that one of the following statements hold.

\vspace{0.2cm} 

\noindent (i) $V_{(1)}$ is a simple Leibniz algebra;

\vspace{0.2cm}

\noindent (ii) $V_{(1)}$ is a semisimple Leibniz algebra and $Ker(D)\cap V_{(0)}=\{a\in V_{(0)}~|~b_0a=0\text{ for all } b\in V_{(1)}\}$. 

\vspace{0.2cm}

\noindent Then $V$ is an indecomposable non-simple vertex algebra. 
\end{prop}

\begin{dfn}(\cite{FLM2}, \cite{LLi}) An ideal of the vertex algebra $V$ is a subspace $I$ such that $u_nw\in I$ and $w_nu\in I$ for $u\in V$, $w\in I$ and $n\in\mathbb{Z}$. 
\end{dfn} 

\noindent Notice that $D(w)=w_{-2}{\bf 1}\in I$. Hence, under the condition that $D(I)\subseteq I$, the left ideal condition $v_nw\in I$ for all $v\in V$, $w\in I$, $n\in\mathbb{Z}$ is equivalent to the right ideal condition $w_mv\in I$ for all $v\in V$, $w\in I$, $m\in\mathbb{Z}$.

\vspace{0.2cm}

\noindent For a subset $S$ of a vertex algebra $V$, we denote by $(S)$ the smallest ideal of $V$ containing $S$. It was shown in Corollary 4.5.10 of \cite{LLi} that 
$$(S)=Span\{~v_nD^i(u)~|~v\in V, n\in\mathbb{Z}, i\geq 0, u\in S\}.$$

\vspace{0.2cm}

\begin{dfn}  For a vertex algebra $V$, we define $C_2(V)=Span\{u_{-2}v~|~u,v\in V\}$. The vertex algebra $V$ is said to satisfy {\em  the $C_2$-condition} if $V/C_2(V)$ is finite dimensional
\end{dfn}

\begin{prop}\cite{Z} For $u,v\in V$, $n\geq 2$, $D(v)\in C_2(V)$ and $u_{-n}v\in C_2(V)$.
\end{prop}

\begin{dfn}\cite{LLi} A $V$-{\em module} is a vector space $W$ equipped with a linear map $Y_W$ from $V$ to $(\End W)[[x,x^{-1}]]$ where $Y_W(v,x)=\sum_{n\in\mathbb{Z}}v_nx^{-n-1}$ for $v\in V$ such that for $u,v\in V$, $w\in W$, 
\begin{eqnarray*}
&&u_nw=0\text{ for $n$ sufficiently large},\\
&&Y_W({\bf 1},x)=1,\\
&&x_0^{-1}\delta\left(\frac{x_1-x_2}{x_0}\right)Y_W(u,x_1)Y_W(v,x_2)-x_0^{-1}\delta\left(\frac{x_2-x_1}{-x_0}\right)Y_W(v,x_2)Y_W(u,x_1)\\
&&=x_2^{-1}\delta\left(\frac{x_1-x_0}{x_2}\right)Y_W(Y(u,x_0)v,x_2).
\end{eqnarray*}
\end{dfn}
\begin{dfn} Let $V=\oplus_{n=0}^{\infty}V_{(n)}$ be a $\mathbb{N}$-graded vertex algebra. A $\mathbb{N}$-graded $V$-module is a $V$-module $M$ equipped with a $\mathbb{N}$-grading $M=\oplus_{n=0}^{\infty}M_{(n)}$ such that $M_{(0)}\neq\{0\}$ and $v_mM_{(n)}\subset M_{(n+p-m-1)}$ for $v\in V_{(p)}$, $p,n\in\mathbb{N}$, $m\in\mathbb{Z}$.
\end{dfn}

\begin{prop}\label{nil}\cite{LLi} Let $V$ be a vertex algebra.

\noindent (i) For $v\in V$, $v$ is weakly nilpotent if and only if $(v_{-1})^r{\bf 1}=0$ for some $r>0$. 

\noindent (ii) Also, if $u\in V$ such that $u_nu=0$ for all $n\geq 0$, then $Y((u_{-1})^r{\bf 1},z)=Y(u,z)^r$ for $r>0$.

\noindent (iii) Let $(W,Y_W)$ be a $V$-module. Then for any weakly nilpotent element $v$ of $V$, with $r>0$ chosen so that $(v_{-1})^r{\bf 1}=0$, we have $Y_W((v_{-1})^r{\bf 1},z)=0$. Also for $u\in U$ such that $u_nu=0$ for all $n\geq 0$, we have $Y_W((u_{-1})^r{\bf 1},z)=Y_W(u,z)^r$ for $r>0$.
\end{prop}

\begin{prop}\label{annihilateVW}\cite{LiY} Let $V$ be a vertex algebra and let $I$ be a (two-sided) ideal generated by a subset $S$. Let $(W,Y_W)$ be a $V$-module  and let $U$ be a generating subspace of $W$ as a $V$-module such that $Y_W(v,x)u=0$ for $v\in S$, $u\in U$. Then $Y_W(v,x)=0$ for $v\in I$.
\end{prop}

\vspace{0.2cm}

\noindent Next, we recall a construction of vertex algebras associated with vertex algebroids in \cite{LiY}. 

\vspace{0.2cm}

\noindent Let $A$ be a commutative associative algebra with identity $\mathfrak{e}$ and let $B$ be a vertex $A$-algebroid. We set $L(A\oplus B)=(A\oplus B)\otimes \mathbb{C}[t,t^{-1}].$ 
Subspaces $L(A)$ and $L(B)$ of $L(A\oplus B)$ are defined in the obvious way. 

\vspace{0.2cm}

\noindent We set $\hat{\partial}=\partial\otimes 1+1\otimes \frac{d}{dt}:L(A)\rightarrow L(A\otimes B).$ We define 
$deg(a\otimes t^n)=-n-1$, $deg(b\otimes t^n)=-n$ for $a\in A, ~b\in B,~n\in\mathbb{Z}$.
Then $L(A\oplus B)$ becomes a $\mathbb{Z}$-graded vector space: 
$$L(A\oplus B)=\oplus_{n\in\mathbb{Z}}L(A\oplus B)_{(n)}$$ where 
$L(A\oplus B)_{(n)}=A\otimes \mathbb{C}t^{-n-1}+B\otimes \mathbb{C}t^{-n}$. Clearly, the subspaces $L(A)$ and $L(B)$ are $\mathbb{Z}$-graded vector spaces as well. In addition, for $n\in \mathbb{N}$, $L(A)_{(n)}=A\otimes  \mathbb{C}t^{-n-1}.$ 
The linear map $\hat{\partial}:L(A)\rightarrow L(A\oplus B)$ is of degree 1. We define a bilinear product $[\cdot,\cdot]$ on $L(A\oplus B)$ as follow:
\begin{eqnarray*}
&&[a\otimes t^m,a'\otimes t^n]=0,\\
&&[a\otimes t^m, b\otimes t^n]=a_0b\otimes t^{m+n},\\
&&[b\otimes t^n,a\otimes t^m]=b_0a\otimes t^{m+n},\\
&&[b\otimes t^m,b'\otimes t^n]=b_0b'\otimes t^{m+n}+m(b_1b)\otimes t^{m+n-1}
\end{eqnarray*}
for $a,a'\in A$, $b,b'\in B$, $m,n\in\mathbb{Z}$. For convenience, we set 
$$\mathcal{L}:=L(A\oplus B)/\hat{\partial}L(A).$$ It was shown in \cite{LiY} that $\mathcal{L}=\oplus_{n\in\mathbb{Z}}\mathcal{L}_{(n)}$ is a $\mathbb{Z}$-graded Lie algebra. Here, 
$$\mathcal{L}_{(n)}=L(A\oplus B)_{(n)}/\hat{\partial}(L(A)_{(n-1)})=(A\otimes \mathbb{C}t^{-n-1}+B\otimes \mathbb{C}t^{-n})/\hat{\partial}(A\otimes\mathbb{C}t^{-n}).$$ In particular, $\mathcal{L}_{(0)}=A\otimes \mathbb{C}t^{-1}+B/\partial A$. 

\vspace{0.2cm}

\noindent Let $\rho:L(A\oplus B)\rightarrow\mathcal{L}$ be a natural linear map defined by $$\rho( u\otimes t^n)=u\otimes t^n+\hat{\partial}L(A).$$ For $u\in A\oplus B$, $n\in\mathbb{Z}$, we set $u(n)=\rho(u\otimes t^n)$ and $u(z)=\sum_{n\in\mathbb{Z}}u(n)z^{-n-1}$. Let $W$ be a $\mathcal{L}$-module. We use $u_W(n)$ or sometimes just $u(n)$ for the corresponding operator on $W$ and we write $u_W(z)=\sum_{n\in\mathbb{Z}}u(n)z^{-n-1}\in (\End W)[[z,z^{-1}]]$. The commutator relations in terms of generating functions are the following:

\begin{eqnarray}
&&[a(z_1),a'(z_2)]=0\label{vaaa'}\\
&&[a(z_1), b(z_2)]=z_2^{-1}\delta\left(\frac{z_1}{z_2}\right)(a_0b)(z_2),\label{vaab}\\
&&[b(z_1),b'(z_2)]=z_2^{-1}\delta\left(\frac{z_1}{z_2}\right)(b_0b')(z_2)+(b_1b')(z_2)\frac{\partial}{\partial z_2}z_2^{-1}\delta\left(\frac{z_1}{z_2}\right)\label{vabb'}
\end{eqnarray}
for $a,a'\in A$, $b,b'\in B$. 
\vspace{0.2cm}

\noindent Next, we define $\mathcal{L}^{\geq 0}=\rho((A\oplus B)\otimes \mathbb{C}[t])\subset \mathcal{L}$, and $\mathcal{L}^{<0}=\rho((A\oplus B)\otimes t^{-1}\mathbb{C}[t^{-1}])\subset \mathcal{L}$. As a vector space, $\mathcal{L}=\mathcal{L}^{\geq 0}\oplus \mathcal{L}^{<0}$. The subspaces $\mathcal{L}^{\geq 0}$ and $\mathcal{L}^{<0}$ are graded sub-algebras.

\vspace{0.2cm}

\noindent We now consider $\mathbb{C}$ as the trivial $\mathcal{L}^{\geq 0}$-module and form the following induced module $$V_{\mathcal{L}}=U(\mathcal{L})\otimes_{U(\mathcal{L}^{\geq 0})}\mathbb{C}.$$ 
In view of the Poincare-Birkhoff-Witt theorem, we have $V_{\mathcal{L}}=U(\mathcal{L}^{<0})$ as a vector space. We may consider $A\oplus B$ as a subspace: $$A\oplus B\rightarrow V_{\mathcal{L}}, \ \  a+b\mapsto a(-1){\bf 1}+b(-1){\bf 1}.$$ 
We assign $deg~\mathbb{C}=0$. Then $V_{\mathcal{L}}=\oplus_{n\in\mathbb{N}}(V_{\mathcal{L}})_{(n)}$ is a restricted $\mathbb{N}$-graded $\mathcal{L}$-module. We set ${\bf 1}=1\in V_{\mathcal{L}}$. 
\begin{prop} \cite{FKRW, MeP} 

\vspace{0.2cm}

\noindent There exists a unique vertex algebra structure on $V_{\mathcal{L}}$ with $Y(u,x)=u(x)$ for $u\in A\oplus B$. In fact, the vertex algebra $V_{\mathcal{L}}$ is a $\mathbb{N}$-graded vertex algebra and it is generated by $A\oplus B$. Furthermore, any restricted $\mathcal{L}$-module $W$ is naturally a $V_{\mathcal{L}}$-module with $Y_W(u,x)=u_W(x)$ for $u\in A\oplus B$. Conversely, any $V_{\mathcal{L}}$-module $W$ is naturally a restricted $\mathcal{L}$-module with $u_W(x)=Y_W(u,x)$ for $u\in A\oplus B$.
\end{prop}

\vspace{0.2cm} 

\noindent Now, we set 
\begin{eqnarray*}
&&E_0=Span\{\mathfrak{e}-{\bf 1},a(-1)a'-a*a'~|~a,a'\in A\}\subset (V_{\mathcal{L}})_{(0)},\\
&&E_1=Span\{a(-1)b-a\cdot b~|~a\in A,b\in\mathfrak{L}\}\subset (V_{\mathcal{L}})_{(1)},\\
&&E=E_0\oplus E_1
\end{eqnarray*}
We define 
$$I_{B}=U(\mathcal{L})\mathbb{C}[D]E.$$ The vector space $I_{B}$ is an $\mathcal{L}$-submodule of $V_{\mathcal{L}}$. We set $$V_{B}=V_{\mathcal{L}}/I_{B}.$$

\begin{prop}\cite{GMS, LiY} \ \ 

\vspace{0.2cm}

\noindent (i) $V_{B}$ is a $\mathbb{N}$-graded vertex algebra such that $(V_{B})_{(0)}=A$ and $(V_{B})_{(1)}=B$ (under the linear map $v\mapsto v(-1){\bf 1}$) and $V_{B}$ as a vertex algebra is generated by $A\oplus B$. Furthermore, for any $n\geq 1$,
\begin{eqnarray*}
&&(V_{B})_{(n)}\\
&&=span\{b_1(-n_1).....b_k(-n_k){\bf 1}~|~b_i\in B,n_1\geq...\geq n_k\geq 1, n_1+...+n_k=n\}.
\end{eqnarray*}

\vspace{0.2cm}

\noindent (ii) A $V_{B}$-module $W$ is a restricted module for the Lie algebra $\mathcal{L}$ with $v(n)$ acting as $v_n$ for $v\in A\oplus B$, $n\in\mathbb{Z}$. Furthermore, the set of $V_{B}$-submodules is precisely the set of $\mathcal{L}$-submodules.
\end{prop}

\vspace{0.2cm}

\noindent Next, we recall definition of Lie algebroid and its module. Also, we will review construction of $V_B$-modules from modules of Lie $A$-algebroid $B/A\partial(A)$. 

\begin{dfn} Let $A$ be a commutative associative algebra. A {\em Lie $A$-algebroid} is a Lie algebra $\mathfrak{g}$ equipped with an $A$-module structure and a module action on $A$ by derivation such that 
$$[u,av]=a[u,v]+(ua)v,\ \ a(ub)=(au)b$$
for all $u,v\in \mathfrak{g},a,b\in A$.

\vspace{0.2cm}

\noindent A {\em module for a Lie $A$-algebroid} $\mathfrak{g}$ is a vector space $W$ equipped with a $\mathfrak{g}$-module structure and an $A$-module structure such that 
$$u(aw)-a(uw)=(ua)w,~a(uw)=(au)w$$ for $a\in A$, $u\in\mathfrak{g}$, $w\in W$. 
\end{dfn}

\begin{prop}\cite{LiY} 
Let $W=\oplus_{n\in\mathbb{N}}W_{(n)}$ be a $\mathbb{N}$-graded $V_{B}$-module with $W_{(0)}\neq\{ 0\}$. Then $W_{(0)}$ is an $A$-module with $a\cdot w=a_{-1}w$ for $a\in A$, $w\in W_{(0)}$ and $W_{(0)}$ is a module for the Lie algebra $B/A\partial(A)$ with $b\cdot w=b_0w$ for $b\in B$, $w\in W_{(0)}$. Furthermore, $W_{(0)}$ equipped with these module structures is a module for Lie $A$-algebroid $B/A\partial A$. If $W$ is graded simple, then $W_{(0)}$ is a simple module for Lie $A$-algebroid $B/A\partial A$.

\end{prop}

\vspace{0.2cm} 

\noindent Now, we set $\mathcal{L}_{\pm }=\oplus_{n\geq 1}\mathcal{L}_{(\pm n)}$ and $\mathcal{L}_{\leq 0}=\mathcal{L}_{-}\oplus \mathcal{L}_{(0)}$. Let $U$ be a module for the Lie algebra $\mathcal{L}_{(0)}(=A\oplus B/\partial (A))$. Then $U$ is an $\mathcal{L}_{(\leq 0)}$-module under the following actions:
$$a(n-1)\cdot u=\delta_{n,0}au,\ \ b(n)\cdot u=\delta_{n,0}u\text{ for }a\in A,b\in B, n\geq 0.$$ 
Next, we form the induced $\mathcal{L}$-module $M(U)=\Ind_{\mathcal{L}_{(\leq 0)}}^{\mathcal{L}}U$. Endow $U$ with degree 0, making $M(U)$ a $\mathbb{N}$-graded $\mathcal{L}$-module. In fact, $M(U)$ is a $V_{\mathcal{L}}$-module. 
We set $$W(U)=span\{v_nu~|~v\in E, ~n\in\mathbb{Z}, ~u\in U\}\subset M(U),$$ and
$$M_{B}(U)=M(U)/U(\mathcal{L})W(U).$$
\begin{prop}\label{simplemodulerelations}\cite{LiY} \ \ 

\vspace{0.2cm} 

\noindent (i) Let $U$ be a module for the Lie algebra $\mathcal{L}_{(0)}$. Then $M_{B}(U)$ is a $V_{B}$-module. If $U$ is a module for the Lie $A$-algebroid $B/A\partial A$ then $(M_{B}(U))_{(0)}=U$.

\vspace{0.2cm} 

\noindent (ii) Let $U$ be a module for the Lie $A$-algebroid $B/A\partial A$. Then there exists a unique maximal graded $U(\mathcal{L})$-submodule $J(U)$  of $M(U)$ with the property that $J(U)\cap U=0$. Moreover, $L(U)=M(U)/J(U)$ is a $\mathbb{N}$-graded $V_{B}$-module such that $L(U)_{(0)}=U$ as a module for the Lie $A$-algebroid $B/A\partial A$. If $U$ is a simple $B/A\partial A$, $L(U)$ is a graded simple $V_{B}$-module.

\vspace{0.2cm}

\noindent (iii) Let $W=\coprod_{n\in\mathbb{N}}W_{(n)}$ be an $\mathbb{N}$-graded simple $V_B$-module with $W_{(0)}\neq 0$. Then $W\cong L(W_{(0)})$. 

\vspace{0.2cm}

\noindent (iv) For any complete set $H$ of representatives of equivalence classes of simple modules for the Lie $A$-algebroid $B/A\partial A$, $\{L(U)~|~U\in H\}$ is a complete set of representatives of equivalence classes of simple $\mathbb{N}$-graded simple $V_B$-modules.
\end{prop}

\section{Proof of Theorem \ref{main1} and Theorem \ref{main2}}

\vspace{0.2cm}

\noindent Let $A$ be a finite-dimensional commutative associative algebra with the identity $\mathfrak{e}$ such that $dim~A\geq 2$. Let $B$ be a finite-dimensional vertex $A$-algebroid such that $A$ is not a trivial $B$-module and $Leib(B)\neq \{0\}$. Let $S$ be its Levi factor such that $S=Span\{e,f,h\}$, $e_0f=h$, $h_0e=2e$, $h_0f=-2f$, and $e_1f=k\mathfrak{e}$. Here, $k\in\mathbb{C}\backslash\{0\}$. Assume that one of the following statements hold.

\vspace{0.2cm}

\noindent (I) $B$ is a simple Leibniz algebra;

\vspace{0.2cm}

\noindent (II) $B$ is a semisimple Leibniz algebra and $Ker(\partial)=\{a\in A~|~b_0a=0\text{ for all }b\in B\}$. 

\vspace{0.2cm}

\noindent We set $A=\mathbb{C}\mathfrak{e}\oplus_{j=1}^l N^j$ where each $N^j$ is an irreducible $sl_2$-submodule of $A$. By Proposition \ref{Bsimple} and Proposition \ref{Bsemisimple}, we have  
 
\vspace{0.2cm}

\noindent (i) $e_1e=f_1f=e_1h=f_1h=0$, $k=1$, $h_1h=2\mathfrak{e}$;

\vspace{0.2cm}

\noindent (ii) $Ker(\partial)=\mathbb{C}\mathfrak{e}$ and $l\geq 1$;

\vspace{0.2cm}

\noindent (iii) For $j\in\{1,...,l\}$ $\dim N^j=2$, and $\dim Leib(B)=2l$; 

\vspace{0.2cm}

\noindent (iv) $A$ is a local algebra. For each $j$, we let $a_{j,0}$ be a highest weight vector of $N^j$ and $a_{j,1}=f_0(a_{j,0})$. Then $\{\mathfrak{e}, a_{j,i}~|~j\in \{1,....,l\},~i\in\{0,1\}\}$ is a basis of $A$, and $\{\partial(a_{j,i})~|~j\in\{1,...,l\},~i\in\{0,1\}\}$ is a basis of $Leib(B)$.  

\vspace{0.2cm}

\noindent Relations among $a_{j,i},e,f,h,\partial(a_{j,i})$ are described below: 
\begin{eqnarray*}
&&a_{j,i}*a_{j',i'}=0,~ a_{j,0}\cdot e=0,~a_{j,1}\cdot e=\partial(a_{j,0}),~a_{j,0}\cdot f=\partial(a_{j,1}),~a_{j,1}\cdot f=0,\\
&&a_{j,0}\cdot h=\partial(a_{j,0}),~a_{j,1}\cdot h=-\partial(a_{j,1}),~a_{j,i}\cdot \partial(a_{j',i'})=0,\\
&&\partial(a_{j,i})_1e=e_0a_{j,i}=(2-i)a_{j,i-1},~\partial(a_{j,i})_1f=f_0a_{j,i}=(i+1)a_{j,i+1},\\
&&\partial(a_{j,i})_1h=h_0a_{j,i}=(1-2i)a_{j,i}.
\end{eqnarray*}

\subsection{Proof of Theorem \ref{main1}} 

\ \ \

\vspace{0.2cm}

\noindent First, we will prove statement $(i)$ of Theorem \ref{main1}.

\begin{lem} $V_B$ is an indecomposable non-simple vertex algebra.
\end{lem}
\begin{proof} Recall that $(V_B)_{(0)}=A$, $(V_B)_{(1)}=B$ and $V_B$ is generated by $A$ and $B$. By Proposition \ref{Vindecomposable}, we can conclude that $V_B$ is an indecomposable non-simple vertex algebra. 
\end{proof}

\noindent Now, we prove statement $(ii)$ of Theorem \ref{main1}. First, we show in Lemma \ref{trivialaction} that if $U$ is an irreducible $B/A\partial(A)$-module such that $a_{j,i}$ acts as zero for all $j\in\{1,...,l\}, i\in\{0,1\}$ then $U$ is an irreducible module for the Lie $A$-algebroid $B/A\partial(A)$. Next, we prove in Lemma \ref{irrdim1dim2} that the converse of the previous statement holds when $U$ has finite dimension (i.e., if $U$ is a finite dimensional irreducible module for the Lie $A$-algebroid $B/A\partial(A)$ then $a_{j,i}$ acts trivial on $U$ for all $j\in\{1,...,l\},i\in\{0,1\}$.) We complete the proof of the statement $(ii)$ of Theorem \ref{main1} in Lemma \ref{simplerelation}.

\begin{lem}\label{trivialaction} Let $U$ be an irreducible $B/A\partial(A)$-module. If $\mathfrak{e}$ acts as a scalar $1$, and $a_{j,i}$ acts trivially on $U$ for all $j\in\{1,...,l\}, i\in\{0,1\}$ then $U$ is an irreducible module for a Lie $A$-algebroid $B/A\partial(A)$. 
\end{lem}
\begin{proof} Let $U$ be an irreducible $B/A\partial(A)$-module. Assume that $\mathfrak{e}$ acts as a scalar $1$, and $a_{j,i}$ acts trivially on $U$ for all $j\in\{1,...,l\}, i\in\{0,1\}$. First, we will show that $U$ is a module for the associative algebra $A$. Let $a=\beta\mathfrak{e}+\sum_{j=1}^l\sum_{i=0}^1\lambda_{j,i}a_{j,i},~a'=\beta'\mathfrak{e}+\sum_{j=1}^l\sum_{i=0}^1\lambda'_{j,i}a_{j,i}\in A$. Here, $\beta,\beta',\lambda_{j,i},\lambda'_{j,i}\in\mathbb{C}$. Since 
\begin{eqnarray*}
(a*a')\cdot w&&=((\beta\mathfrak{e}+\sum_{j=1}^l\sum_{i=0}^1\lambda_{j,i}a_{j,i})*(\beta'\mathfrak{e}+\sum_{j=1}^l\sum_{i=0}^1\lambda'_{j,i}a_{j,i}))\cdot w\\
&&=(\beta a'+\beta'\sum_{j=1}^l\sum_{i=0}^1\lambda_{j,i}a_{j,i})\cdot w=(\beta\beta') w,\text{ and }\\
a\cdot (a'\cdot w)&&=a\cdot (\beta'w)=\beta\beta' w\text{ for all }w\in U,
\end{eqnarray*}
we can conclude that $U$ is a module for the associative algebra $A$. 

\vspace{0.2cm}

\noindent Now, we will show that $U$ is a module for the Lie $A$-algebroid $B/A\partial(A)$. It is enough to show that for $a\in A$, $u\in B/A\partial(A)$, $w\in U$, $u_0(a\cdot w)-a\cdot (u_0w)=(u_0 a)\cdot w$ and $a\cdot (u_0w)=(a\cdot u)_0w$. Recall that for $a,a'\in A$, $b\in B$, $(a\cdot v)_0a'=-a'_0(a\cdot v)=-a*(a'_0v)=a*(v_0a')$. Consequently, $$(\alpha\cdot\partial(\alpha'))_0a'=\alpha*(\partial(\alpha')_0a')=0\text{ for all }\alpha,\alpha',a'\in A.$$ We let $u=\gamma_e e+\gamma_f f+\gamma_h h+A\partial(A)\in B/A\partial(A)$. Here, $\gamma_e,\gamma_f,\gamma_h\in\mathbb{C}$. Observe that
\begin{eqnarray*}
u_0 a&&=(\gamma_e e+\gamma_f f+\gamma_h h+A\partial (A))_0 (\beta\mathfrak{e}+\sum_{j=1}^l\sum_{i=0}^1\lambda_{j,i}a_{j,i})\\
&&=\gamma_e\sum_{j=1}^l\lambda_{j,1}a_{j,0}+\gamma_f\sum_{j=1}^l\lambda_{j,0}a_{j,1}+\gamma_h\sum_{j=1}^l\lambda_{j,0}a_{j,0}+\gamma_h\sum_{j=1}^l\lambda_{j,1}(-a_{j,1}).
\end{eqnarray*}
Hence, $(u_0 a)\cdot w=0$ for all $w\in U$.
Since 
\begin{eqnarray*}
u_0(a\cdot w)-a\cdot(u_0 w)&&=u_0((\beta\mathfrak{e}+\sum_{j=1}^l\sum_{i=0}^1\lambda_{j,i}a_{j,i})\cdot w)-(\beta\mathfrak{e}+\sum_{j=1}^l\sum_{i=0}^1\lambda_{j,i}a_{j,i})\cdot (u_0 w)\\
&&=u_0 (\beta w)-\beta ( u_0w)\\
&&=0
\end{eqnarray*}
we can conclude immediately that $$u_0(a\cdot w)-a\cdot(u_0w)=(u_0a)\cdot w\text{ for all }w\in U.$$ 
Recall that for $j\in\{1,...,l\}, i\in\{0,1\}$, we have $a_{j,i} \cdot v\in \partial(A)$ for all $v\in B$. It follows that 
$$(a\cdot u)_0 w=(\beta u)_0 w=\beta u_0w\text{ for all }w\in U.$$ Moreover, we have
\begin{eqnarray*}
a\cdot(u_0 w)=\beta (u_0w)=(a\cdot u)_0 w\text{ for all }w\in U.
\end{eqnarray*} 
Therefore, $U$ is a module for the Lie $A$-algebroid $B/A\partial(A)$. 

\vspace{0.2cm}

\noindent Next, we will show that $U$ is an irreducible module for the Lie $A$-algebroid $B/A\partial(A)$. Let $N$ be a nonzero Lie $A$-algebroid $B/A\partial(A)$-submodule of $U$. Then $N$ is a $B/A\partial(A)$-submodule of $U$. Since $U$ is an irreducible $B/A\partial(A)$, we can conclude that $N=U$ and $U$ is an irreducible module for the Lie $A$-algebroid $B/A\partial(A)$. This completes the proof of this Lemma. 
\end{proof}

\begin{lem}\label{irrdim1dim2} Let $W$ be a finite dimensional irreducible module of the Lie $A$-algebroid $B/A\partial(A)$. Then for $j\in\{1,...,l\},i\in\{0,1\}$, $a_{j,i}$ acts trivially on $W$. In addition, $W$ is an irreducible $sl_2$-module. 
\end{lem}
\begin{proof} Let $W$ be a finite dimensional irreducible module of the Lie $A$-algebroid $B/A\partial(A)$. First, we will show that if $W$ has dimension 1 then $W$ is a trivial $sl_2$-module such that for $j\in\{1,...,l\}, i\in\{0,1\}$, $a_{j,i}$ acts as zero on $W$. For simplicity, we assume that $W=\mathbb{C}b$ for some $b\in W$. Clearly, $W$ is a trivial $sl_2$-module. We set $a_{j,i}\cdot b=\beta_{j,i}b$. Here, $\beta_{j,i}\in\mathbb{C}$. Since 
\begin{eqnarray*}
&&h_{0}(a_{j,i}\cdot b)=(h_{0}a_{j,i})\cdot b+(a_{j,i})\cdot (h_{0}b)=(h_{0}a_{j,i})\cdot b,\text{ and }\\
&&h_{0}(a_{j,i}\cdot b)=h_{0}(\beta_{j,i}b)=0,
\end{eqnarray*}
we then have that 
\begin{eqnarray*}
&&0=(h_{0}a_{j,0})\cdot b=a_{j,0}\cdot b=\beta_{j,0} b\text{ and }\\
&&0=(h_{0}a_{j,1})\cdot b=-a_{j,1}\cdot b=-\beta_{j,1}b.
\end{eqnarray*} 
Therefore, $\beta_{j,i}=0$ and $a_{j,i}$ acts as zero on $W$ for all $j\in\{1,...,l\}$ and $i\in\{0,1\}$. 

\vspace{0.2cm} 

\noindent Next, we assume that $W$ has dimension 2. Hence, $W$ is either a direct sum of two one-dimensional trivial $sl_2$-modules or $W$ is a two-dimensional irreducible $sl_2$-module. Suppose that $W=\mathbb{C}b_1\oplus\mathbb{C}b_2$ where $\mathbb{C}b_1$ and $\mathbb{C}b_2$ are trivial $B/A\partial(A)$-modules. For $j\in\{1,...,l\}$, $i\in\{0,1\}$, we set $a_{j,i}\cdot b_1=\beta_{j,i,1}b_1+\beta_{j,i,2}b_2$. Since 
$$h_0(a_{j,i}\cdot b_1)=(h_0(a_{j,i}))\cdot b_1\text{ and }h_0(\beta_{j,i,1}b_1+\beta_{j,i,2}b_2)=0,$$
we then have that $a_{j,i}\cdot b_1=0$ for all $j\in\{1,...,l\}$, $i\in\{0,1\}$. Consequently, $\mathbb{C}b_1$ is an irreducible $A$-Lie algebroid $B/A\partial(A)$. This contradicts with our assumption that $W$ is an irreducible Lie $A$-algebroid $B/A\partial(A)$. Therefore, $W$ is a two-dimensional irreducible $sl_2$-module. Let $w_0$ be a highest weight vector of $W$ of weight 1 and let $w_1=f_0(w_0)$. The set $\{w_0, w_1\}$ is a basis of $W$. For $j\in\{1,...,l\}$, $i\in\{0,1\}$, we set 
$$a_{j,i}\cdot (w_0)=\alpha_{j,i,0} w_0+\alpha_{j,i,1} w_1.$$ Here, $\alpha_{j,i,0},\alpha_{j,i,1}\in\mathbb{C}.$ Notice that
\begin{eqnarray*}
&&h_0(a_{j,i}\cdot w_0)=(h_0a_{j,i})\cdot w_0+a_{j,i}\cdot (h_0(w_0))=(h_0a_{j,i})\cdot w_0+a_{j,i}\cdot w_0,\text{ and }\\
&&h_0(\alpha_{j,i,0} w_0+\alpha_{j,i,1} w_1)=\alpha_{j,i,0}w_0+\alpha_{j,i,1}(-w_1).
\end{eqnarray*} 
So, we have 
\begin{eqnarray*}
&&2a_{j,0}\cdot w_0=(h_0a_{j,0})\cdot w_0+a_{j,0}\cdot w_0=\alpha_{j,0,0}w_0+\alpha_{j,0,1}(-w_1),\text{ and }\\
&&0=(h_0a_{j,1})\cdot w_0+a_{j,1}\cdot w_0=\alpha_{j,1,0}w_0+\alpha_{j,1,1}(-w_1).
\end{eqnarray*} Therefore, for $j\in\{1,...,l\}$, $i\in\{0,1\}$, we have $\alpha_{j,i,0}=\alpha_{j,i,1}=0$ and $a_{j,i}\cdot w_0=0$. Since $a_{j,0}\cdot f=\partial(a_{j,1})$ and $a_{j,1}\cdot f=0$, these imply that 
\begin{eqnarray*}
a_{j,i}\cdot w_1&&=a_{j,i}\cdot (f_0w_0)=(a_{j,i}\cdot f)_0 (w_0)=0
\end{eqnarray*} 
Consequently, if $\dim~ W=2$ then $W$ is an irreducible $B/A\partial(A)$-module such that for $j\in\{1,...,l\},i\in\{0,1\}$, $a_{j,i}$ acts trivially on $W$.

\vspace{0.2cm} 

\noindent Now, we study that case when $\dim W\geq 3$. Suppose that $W$ contains a nonzero proper $B/A\partial(A)$-submodule (i.e., we consider $W$ as a module for the Lie algebra $B/A\partial(A)$). Since $B/A\partial(A)$ is semisimple, this implies that there exist irreducible $B/A\partial(A)$-modules $U_1,...,U_t$ such that $W=\oplus_{i=1}^t U_t$. For each $i\in\{1,...,t\}$, we let $w_{i,0}$ be a highest weight vector of $U_i$ of weight $m_i$. Also, we set $w_{i,s}=\frac{1}{s!}(f(0))^sw_{i,0}$. Clearly, $\{w_{i,0},....,w_{i,m_i}\}$ form a basis of $U_{i}$. Let $j\in\{1,...,l\}$, $1\leq s\leq m_i$. Since $a_{j,0}\cdot f=\partial(a_{j,1})$ and $a_{j,1}\cdot f=0$, we then have that 
\begin{eqnarray*}
a_{j,0}\cdot w_{i,s}&&=a_{j,0}\cdot \left(\frac{1}{s!}(f_0)^s w_{i,0}\right)=a_{j,0}\cdot \left(f_0(\frac{1}{s!}(f_0)^{s-1}w_{i,0})\right)\\
&&=(a_{j,0}\cdot f)_0\left(\frac{1}{s!}(f_0)^{s-1}w_{i,0}\right)=0,\end{eqnarray*}
and 
$$a_{j,1}\cdot w_{i,s}=a_{j,1}\cdot \left(\frac{1}{s!}(f_0)^sw_{i,0}\right)=(a_{j,1}\cdot f)_0\left(\frac{1}{s!}(f_0)^{s-1}w_{i,0}\right)=0.$$ 
Since $h_0(a_{j,0}\cdot w_{i,0})=(h_0a_{j,0})\cdot w_{i,0}+a_{j,0}\cdot (h_0w_{i,0})=(m_i+1)a_{j,0}\cdot w_{i,0}$ and 
\begin{eqnarray*}
h_0(a_{j,0}\cdot w_{i,0})&&=e_0f_0(a_{j,0}\cdot w_{i,0})-f_0e_0(a_{j,0}\cdot w_{i,0})\\
&&=e_0(f_0(a_{j,0})\cdot w_{i,0}+a_{j,0}\cdot (f_0 w_{i,0}))-f_0((e_0a_{j,0})\cdot w_{i,0}+a_{j,0}\cdot e_0w_{i,0})\\
&&=e_0(a_{j,1}\cdot w_{i,0}+a_{j,0}\cdot w_{i,1})\\
&&=e_0(a_{j,1}\cdot w_{i,0})\\
&&=(e_0a_{j,1})\cdot w_{i,0}+a_{j,1}\cdot (e_0w_{i,0})\\
&&=a_{j,0}\cdot w_{i,0},
\end{eqnarray*} 
we can conclude that $m_i=0$ and $U_i$ is a trivial $B/A\partial(A)$-module. Moreover, $\{w_{1,0},...,w_{t,0}\}$ is a basis of $W$. For $j\in\{1,...,l\}$, $i\in\{1,...,t\}$, we set $a_{j,0}\cdot w_{i,0}=\sum_{p=1}^t\alpha_p w_{p,0}$, and $a_{j,1}\cdot w_{i,0}=\sum_{p=1}^t\gamma_p w_{p,0}$ where $\alpha_p,\gamma_p\in\mathbb{C}$. Since 
\begin{eqnarray*}
&&h_0(a_{j,0}\cdot w_{i,0})=(h_0a_{j,0})\cdot w_{i,0}+a_{j,0}\cdot h_0w_{i,0}=a_{j,0}\cdot w_{i,0},\\
&&h_0(\sum_{p=1}^t\alpha_p w_{p,0} )=0,\\
&&h_0(a_{j,1}\cdot w_{i,0})=(h_0a_{j,1})\cdot w_{i,0}+a_{j,1}\cdot h_0w_{i,0}=-a_{j,1}\cdot w_{i,0},\\
&&h_0(\sum_{p=1}^t\gamma_p w_{p,0} )=0,
\end{eqnarray*}
we can conclude that $a_{j,0}\cdot w_{i,0}=0=a_{j,1}\cdot w_{i,0}$ for all $j\in\{1,...,l\}$, $i\in\{1,...,t\}$. Moreover, each $U_i$ is an irreducible module for the Lie $A$-algebroid $B/A\partial(A)$. This is a contradiction. Hence, $W$ is an irreducible $B/A\partial(A)$-module.

\vspace{0.2cm}

\noindent Now, we let $u_0$ be the highest weight vector of $W$ with weight $m$ and for $i\in\{1,...,m\}$ we let $u_i=\frac{1}{i!}(f(0))^iu_0$. We have
$$a_{j,0}\cdot u_{i}=a_{j,0}\cdot \left(\frac{1}{i!}(f_0)^i u_{0}\right)=a_{j,0}\cdot \left(f_0(\frac{1}{i!}(f_0)^{i-1}u_{0})\right)=(a_{j,0}\cdot f)_0\left(\frac{1}{i!}(f_0)^{i-1}u_{0}\right)=0,$$ 
and 
$$a_{j,1}\cdot u_{i}=a_{j,1}\cdot \left(\frac{1}{i!}(f_0)^iu_{0}\right)=(a_{j,1}\cdot f)_0\left(\frac{1}{i!}(f_0)^{i-1}u_{0}\right)=0.$$ 
Next, we set $a_{j,0}\cdot u_0=\sum_{q=0}^{m}\alpha_qu_q$. Since 
\begin{eqnarray*}
&&h_0(a_{j,0}\cdot u_0)=(m+1)(a_{j,0}\cdot u_0)=\sum_{q=0}^m\alpha_q(m+1)u_q\text{ and }\\
&&h_0(\sum_{q=0}^m\alpha_qu_q)=\sum_{q=0}^m(m-2q)\alpha_q u_q,
\end{eqnarray*} 
we can conclude that $\alpha_q(m+1)=\alpha_q(m-2q)\text{ for all }0\leq q\leq m$. If $\alpha_q\neq 0$, we have $m+1=m-2q$ which is impossible. Therefore, for all $0\leq q\leq m$, $\alpha_q=0$. Consequently, we have $a_{j,0}\cdot u_0=0$. Moreover, we have  
$$a_{j,1}\cdot u_0=(f_0a_{j,0})\cdot u_0=f_0(a_{j,0}\cdot u_0)-a_{j,0}\cdot (f_0u_0)=-a_{j,0}\cdot u_1=0.$$ 
Hence, $W$ is an irreducible $B/A\partial(A)$-module such that for $j\in\{1,...,l\},i\in\{0,1\}$, $a_{j,i}$ acts trivially on $W$. This completes the proof of this Lemma.
\end{proof}


\begin{lem}\label{simplerelation} The set of representatives of equivalence classes of  finite-dimensional simple $sl_2$-modules is equivalent to the set of representatives of equivalence classes of $\mathbb{N}$-graded simple $V_B$-modules $N=\oplus_{n=0}^{\infty}N_{(n)}$ such that $dim~N_{(0)}<\infty$.
\end{lem}
\begin{proof} By Lemma \ref{trivialaction}, every finite dimensional irreducible $B/A\partial (A)$-module is an irreducible module for the Lie $A$-algebroid $B/A\partial(A)$. By Lemma \ref{irrdim1dim2}, the set of representatives of equivalence classes of finite dimensional simple modules for the Lie $A$-algebraoid $B/A\partial(A)$ equals the set of representatives of equivalence classes of finite dimensional simple modules for the Lie algebra $B/A\partial(A)$. By Proposition \ref{simplemodulerelations}, we can conclude that the set of representatives of equivalence classes of  finite-dimensional simple $sl_2$-modules is equivalent to the set of representatives of equivalence classes of $\mathbb{N}$-graded simple $V_B$-modules $N=\oplus_{n=0}^{\infty}N_{(n)}$ such that $dim~N_{(0)}<\infty$.
\end{proof} 
\noindent This completes the proof of statment $(ii)$ of Theorem \ref{main1}.

\subsection{Proof of Theorem \ref{main2}}

\ \ \

\vspace{0.2cm}

\noindent First, we will prove statement $(i)$ of Theorem \ref{main2}. 
Let $(e(-1)e)$ be an ideal of $V_B$ that is generated by $e(-1)e$. 
\begin{lem}\label{ee} $(e(-1)e)\cap A=\{0\}$ and $(e(-1)e)\cap B=\{0\}$.
\end{lem}
\begin{proof}  First, we will show that $v_{(\deg v)+1}e(-1)e=0$ and $v_{\deg v}e(-1)e=0$ for every homogeneous $v\in V_B$.  We will separate our proof into several steps. For the first step, we will show that for $a\in A$, $$a(n)e(-1)e=0\text{ for all }n\geq 0.$$ Recall that for $a\in A$, $b \in B$, we have $b(-1)a=a(-1)b-D(a(0)b)=a\cdot b-D(a_0b)$. Let $a=\alpha{\bf 1}+\sum_{j=1}^l\sum_{i=0}^1\alpha_{j,i}a_{j,i}$. Here, $\alpha,\alpha_{j,i}\in\mathbb{C}$. It is straightforward to show that 
\begin{eqnarray*}
&&a(0)e(-1)e\\
&&=e(-1)a_0e+(a_0e)(-1)e\\
&&=-e(-1)e_0a-(e_0a)(-1)e\\
&&=-e(-1)(\sum_{j=1}^l\alpha_{j,1}a_{j,0})-\sum_{j=1}^l\alpha_{j,1}(a_{j,0})(-1)e\\
&&=-e(-1)(\sum_{j=1}^l\alpha_{j,1}a_{j,0})-\sum_{j=1}^l\alpha_{j,1}(a_{j,0})\cdot e\\
&&=-(\sum_{j=1}^l\alpha_{j,1}a_{j,0})\cdot e+D((\sum_{j=1}^l\alpha_{j,1}a_{j,0})_0e)\\
&&=0,
\end{eqnarray*}
and $a(1)e(-1)e=e(-1)a(1)e+(a_0e)(0)e=e_0(e_0a)=e_0\left(\sum_{j=1}^l\alpha_{j,1}a_{j,0}\right)=0$. Hence, $$a(n)e(-1)e=0\text{ for all }n\geq 0.$$ 
For the second step, we will show that for $b\in B$, $$(b(-1){\bf 1})_ne(-1)e=b(n)e(-1)e=0,~(b(-m){\bf 1})_{m+1}e(-1)e=0,\text{ and }(b(-m){\bf 1})_me(-1)e=0$$ for all $m\geq 2$, $n\geq 1$. Let $b=\beta_e e+\beta_f f+\beta_h h+\sum_{j=1}^l\sum_{i=0}^1\beta_{j,i}\partial(a_{j,i})\in B$.  Here, $\beta_e,\beta_f,\beta_h,\beta_{j,i}\in\mathbb{C}$. Since
\begin{eqnarray*}
b_0e&&=(\beta_e e+\beta_f f+\beta_h h+\sum_{j=1}^l\sum_{i=0}^1\beta_{j,i}\partial(a_{j,i}) )_0e\\
&&=\beta_f(-h)+\beta_h(2e)\text{ and }\\
b_1e&&=(\beta_e e+\beta_f f+\beta_h h+\sum_{j=1}^l\sum_{i=0}^1\beta_{j,i}\partial(a_{j,i}) )_1e\\
&&=\beta_f{\bf 1}+\sum_{j=1}^l\beta_{j,1}a_{j,0}
\end{eqnarray*}
we then have that
\begin{eqnarray*}
b(1)e(-1)e&&=e(-1)b(1)e+(b_0e)(0)e+(b_1e)(-1)e\\
&&=e(-1)(b_1e)+(\beta_f(-h)+\beta_h(2e))_0e+(\beta_f{\bf 1}+\sum_{j=1}^l\beta_{j,1}a_{j,0})(-1)e\\
&&=e(-1)(\beta_f{\bf 1}+\sum_{j=1}^l\beta_{j,1}a_{j,0})+\beta_f(-2)e+\beta_f e\\
&&=\sum_{j=1}^l\beta_{j,1}(a_{j,0}\cdot e+\partial((a_{j,0})_0e)\\
&&=0,\text{ and }\\
b(2)e(-1)e&&=e(-1)b(2)e+(b_0e)(1)e+2(b_1e)(0)e\\
&&=( \beta_f(-h)+\beta_h(2e))_1e+2(\beta_f{\bf 1}+\sum_{j=1}^l\beta_{j,1}a_{j,0} )_0e\\
&&=0.
\end{eqnarray*}
Hence, $(b(-1){\bf 1})_ne(-1)e=b(n)e(-1)e=0$ for all $n\geq 1$. Let $m\geq 2$. For $t\geq 1$, 
\begin{eqnarray*}
&&(b(-m){\bf 1})_te(-1)e\\
&&=\sum_{i\geq 0}(-1)^i{-m\choose i}(b(-m-i){\bf 1}(t+i)-(-1)^{-m}{\bf 1}(-m+t-i)b(i))e(-1)e\\
&&=-(-1)^{-m}{\bf 1}(-m+t)b(0)e(-1)e.
\end{eqnarray*}
This implies that $(b(-m){\bf 1})_{m+1}e(-1)e=0$ and $(b(-m){\bf 1})_{m}e(-1)e=0$.

\vspace{0.2cm} 

\noindent Recall that for $n\geq 1$, 
\begin{eqnarray*}
&&(V_{B})_{(n)}\\
&&=span\{b_1(-n_1).....b_k(-n_k){\bf 1}~|~b_i\in B,n_1\geq...\geq n_k\geq 1, n_1+...+n_k=n\}.
\end{eqnarray*}
If $v\in V_B$ is of the form $b_1(-n_1).....b_k(-n_k){\bf 1}$ where $b_i\in\{e,f,h, \partial(a_{j,i})~|~j\in\{1,...,l\}, i\in\{0,1\}\}$, we say that $v$ is a monomial vector that has length $k$. For the third step, we will show that if $b\in B$, and $v$ is a monomial vector of length $k$ then either $b(0)v=0$ or $b(0)v$ is a sum of monomial vectors that have length $k$. Clearly, if $b\in \partial(A)$ then $b_0v=0$.  Notice that 
\begin{eqnarray*}
&&b(0)b^1(-n_1){\bf 1}=(b_0b^1)(-n_1){\bf 1}\\
&&b(0)b^1(-n_1)b^2(-n_2){\bf 1}=b^1(-n_1)(b_0b^2)(-n_2){\bf 1}+(b_0b^1)(-n_1)b^2(-n_2){\bf 1}
\end{eqnarray*} 
If $b_0b^1=0$ then $b(0)b^1(-n_1){\bf 1}=0$. If $b_0b^1\neq 0$ then $b(0)b^1(-n_1){\bf 1}$ is a sum of monomial vectors of length 1. Similarly, if $b_0b^2=0=b_0b^1$ then $b(0)b^1(-n_1)b^2(-n_2){\bf 1}=0$. Otherwise, $b(0)b^1(-n_1)b^2(-n_2){\bf 1}$ is a sum of monomial vectors of length 2. 
Now, we assume that for monomial vectors $w$ of length $t$, either $b(0)w=0$ or $b(0)w$ are the sum of monomial vectors of length $t$. Since 
\begin{eqnarray*}
&&b(0)b_1(-n_1).....b_{t+1}(-n_{t+1}){\bf 1}\\
&&=b_1(-n_1)b(0)b_2(-n_2).....b_{t+1}(-n_{t+1}){\bf 1}\\
&&\ \ \ \ \ +(b_0b^1)(-n_1)b_2(-n_2).....b_{t+1}(-n_{t+1}){\bf 1},
\end{eqnarray*}
by induction hypothesis, we can conclude that either $b(0)b_1(-n_1).....b_{t+1}(-n_{t+1}){\bf 1}=0$ or $b(0)b_1(-n_1).....b_{t+1}(-n_{t+1}){\bf 1}$ is a sum of monomial vectors of length $t+1$. 

\vspace{0.2cm}

\noindent For the fourth step, we will show that for every monomial vector $v\in V_B$, $v_{(\deg v)+1}e(-1)e=0$ and $v_{\deg v}e(-1)e=0$. We will use an induction on the length of monomial vectors to prove this statement. By the first step and the second step, we can conclude immediately that if $v$ is a monomial vector of length $q$ where $0\leq q\leq 1$ then $v_{(\deg v)+1}e(-1)e=0$ and $v_{(\deg v)}e(-1)e=0$. Now, we assume that for any monomial vector $v$ of length $k\leq t$, $v_{(\deg v)+1}e(-1)e=0$ and $v_{(\deg v)}e(-1)e=0$. For $i\in\{1,...,t+1\}$, we let $n_i$ be a positive integer, $b^i\in B$. We set $n=n_1+n_2+...+n_t+n_{t+1}+1$. Notice that
\begin{eqnarray*}
&&(b^1(-n_1)b^2(-n_2)...b^t(-n_t)b^{t+1}(-n_{t+1}){\bf 1})_ne(-1)e\\
&&=\sum_{i\geq 0}(-1)^i{-n_1\choose i}(b^1(-n_1-i)(b^2(-n_2)...b^t(-n_t)b^{t+1}(-n_{t+1}){\bf 1})_{n+i}\\
&&\ \ \ \ \ \ \ -(-1)^{n_1}(b^2(-n_2)...b^t(-n_t)b^{t+1}(-n_{t+1}){\bf 1})_{-n_1+n-i}b^1(i))e(-1)e\\
&&=-(-1)^{n_1}(b^2(-n_2)...b^{t+1}(-n_{t+1}){\bf 1})_{-n_1+n}b^1(0)e(-1)e\ \ (\text{by induction hypothesis})\\
&&=(-1)^{n_1+1}(b^1(0)(b^2(-n_2)...b^{t+1}(-n_{t+1}){\bf 1})_{-n_1+n}e(-1)e\\
&&\ \ \ \ \ \ -(b^1(0)b^2(-n_2)...b^{t+1}(-n_{t+1}){\bf 1})_{-n_1+n}e(-1)e),\\
\end{eqnarray*}
For simplicity, we set $$u=b^2(-n_2)...b^{t+1}(-n_{t+1}){\bf 1}\text{ and }w= b^1(0)b^2(-n_2)...b^{t+1}(-n_{t+1}){\bf 1}.$$ 
Note that the vector $u$ is a monomial vector of length $t$, the vector $w$ is a sum of monomial vectors of length $t$, and $$-n_1+n=n_2+....+n_{t+1}+1=(\deg u)+1=(\deg w)+1.$$ By induction hypothesis, we can conclude that 
$$(b^1(-n_1)b^2(-n_2)...b^t(-n_t)b^{t+1}(-n_{t+1}){\bf 1})_ne(-1)e=0.$$

\vspace{0.2cm}

\noindent Next, we assume that $n=n_1+n_2+...+n_{t+1}$. We have
\begin{eqnarray*}
&&(b^1(-n_1)b^2(-n_2)...b^{t+1}(-n_{t+1}){\bf 1})_ne(-1)e\\
&&=\sum_{i\geq 0}(-1)^i{-n_1\choose i}(b^1(-n_1-i)(b^2(-n_2)...b^{t+1}(-n_{t+1}){\bf 1})_{n+i}\\
&&\ \ \ \ \ \ \ -(-1)^{n_1}(b^2(-n_2)...b^{t+1}(-n_{t+1}){\bf 1})_{-n_1+n-i}b^1(i))e(-1)e\\
&&=\delta_{n_1,1}b^1(-n_1)(b^2(-n_2)...b^{t+1}(-n_{t+1}){\bf 1})_{n}e(-1)e\\
&&\ \ \ \ \ \ \ -(-1)^{n_1}(b^2(-n_2)...b^{t+1}(-n_{t+1}){\bf 1})_{-n_1+n}b^1(0)e(-1)e\\
&&=\delta_{n_1,1}b^1(-n_1)(b^2(-n_2)...b^{t+1}(-n_{t+1}){\bf 1})_{n}e(-1)e\\
&&\ \ \ \ -(-1)^{n_1}\{b^1(0)(b^2(-n_2)...b^{t+1}(-n_{t+1}){\bf 1})_{-n_1+n}\\
&&\ \ \ \ \ -(b^1(0)b^2(-n_2)...b^{t+1}(-n_{t+1}){\bf 1})_{-n_1+n}\}e(-1)e.
\end{eqnarray*} 
We set $p=b^2(-n_2)...b^{t+1}(-n_{t+1}){\bf 1}$ and $r=b^1(0)b^2(-n_2)...b^{t+1}(-n_{t+1}){\bf 1}$.
Notice that when $n_1=1$, $p_n=p_{(\deg p)+1}$. Also, $p_{-n_1+n}=p_{\deg p}$, $r$ is a sum of monomial vectors of length $t$ and $r_{-n_1+n}=r_{\deg r}$. By induction hypothesis, we can conclude that 
$$(b^1(-n_1)b^2(-n_2)...b^{t+1}(-n_{t+1}){\bf 1})_ne(-1)e=0.$$ 
Hence, for any homogeneous monomial vector $v$ of length $k$, $v_{(\deg v)+1}e(-1)e=0$ and $v_{\deg v}e(-1)e=0$. This completes the fourth step. 

\vspace{0.2cm} 

\noindent Because $v_{(\deg v)+1}e(-1)e=0$ and $v_{\deg v}e(-1)e=0$ for any homogeneous monomial vector $v$ of any length $k$, we can conclude further that for any homogeneous vector $u\in V_B$, 
\begin{eqnarray*}
&&u_{(\deg u)+1}e(-1)e=0,\text{ and }\\
&&u_{\deg u}e(-1)e=0.
\end{eqnarray*}  
Moreover, for any homogeneous vector $v\in V_B$, $t\in\mathbb{Z}$, we have $v_te(-1)e\in\oplus_{n=2}^{\infty}(V_B)_{(n)}$. This implies that for $v\in V_B$, $t\in\mathbb{Z}$, we have $v_te(-1)e\in\oplus_{n=2}^{\infty}(V_B)_{(n)}$.
 
 \vspace{0.2cm}
 
 \noindent Next, we will show that for $i\geq 1$, $v_tD^i e(-1)e\in \oplus_{n=2}^{\infty}(V_B)_{(n)}$ for all $v\in V_B$, $t\in\mathbb{Z}$. Clearly,
$$v_tDe(-1)e=Dv_te(-1)e+tv_{t-1}e(-1)e\in \oplus_{n=2}^{\infty}(V_B)_{(n)}.$$ Now, let us assume that 
$v_tD^je(-1)e\in \oplus_{n=2}^{\infty}(V_B)_{(n)}$ for all $v\in V_B$, $t\in \mathbb{Z}$. Since 
$$v_tD^{j+1}e(-1)e=Dv_tD^je(-1)e+tv_{t-1}D^je(-1)e,$$ 
we can conclude immediately that $v_tD^{j+1}e(-1)e\in\oplus_{n=2}^{\infty}(V_B)_{(n)}.$ Hence, $$v_tD^i e(-1)e\in \oplus_{n=2}^{\infty}(V_B)_{(n)}\text{ for all }i\geq 0,~v\in V_B,~t\in\mathbb{Z}.$$ 
This implies that $(e(-1)e)\cap (A\oplus B)=\{0\}$.
\end{proof}

\vspace{0.2cm}

\noindent We set $$\overline{V_B}=V_B/(e(-1)e).$$ 
\begin{prop} $\overline{V_B}=\oplus_{n=0}^{\infty}(\overline{V_B})_{(n)}$ is an indecomposable non-simple $\mathbb{N}$-graded vertex algebra such that $(\overline{V_B})_{(0)}=A$ and $(\overline{V_B})_{(1)}=B$. 
\end{prop}
\begin{proof} Since $\overline{V_B}=\oplus_{n=0}^{\infty}(\overline{V_B})_{(n)}$ is a $\mathbb{N}$-graded vertex algebra such that $(\overline{V_B})_{(0)}=A$ and $(\overline{V_B})_{(1)}=B$, by Proposition \ref{Vindecomposable}, we can conclude that $\overline{V_B}=\oplus_{n=0}^{\infty}(\overline{V_B})_{(n)}$ is an indecomposable non-simple $\mathbb{N}$-graded vertex algebra.
\end{proof}
\noindent This completes the proof of statement $(i)$ of Theorem \ref{main2}. 

\vspace{0.2cm}

\noindent To prove statement $(ii)$ and statement $(iii)$ of Theorem \ref{main2}, we need  to use properties of vertex operator algebras associated with a certain type of rank one lattices, and vertex operator algebras associated with highest weight representations of affince Lie algebras. We provide background material on these topics in Appendices.  

\vspace{0.2cm}

\noindent Let $\hat{S}=S\otimes \mathbb{C}[t,t^{-1}]\oplus \mathbb{C}c$ be the affine Lie algebra where $c$ is central and
$$[u\otimes t^m,v\otimes^n]=[u,v]\otimes t^{m+n}+m\langle\langle u,v\rangle\rangle \delta_{m+n,0}c.$$ Here, $\langle\langle~,~\rangle\rangle$ is a symmetric invariant bilinear form of $S$ such that $\langle\langle e,f\rangle\rangle=1$, $\langle\langle h,h\rangle\rangle=2$ and $\langle\langle e,e\rangle\rangle=\langle\langle f,f\rangle\rangle=\langle\langle e,h\rangle\rangle=\langle\langle f,h\rangle\rangle=0$.  
The generalized Verma $\hat{S}$-module $M_{S}(k,0)$ is a vertex operator algebra (see Appendices for the construction of the vertex operator algebra $M_{S}(k,0)$ and its properties).

\vspace{0.2cm}

\noindent For $u\in \overline{V_B}$, we set $Y_{\overline{V_B}}(u,z)=\sum_{n\in\mathbb{Z}}u[n]z^{-n-1}$. Since $S=Span\{e,f,h\}$ is a subset of $(\overline{V_B})_{(1)}$ and $S$ is a Lie algebra with a symmetric invariant bilinear form $\langle~,~\rangle:S\times S\rightarrow\mathbb{C}$ such that $\langle s,s'\rangle\mathfrak{e}=s[1]s'$, the map $\hat{S}\rightarrow \End(\overline{V_B}): s\otimes t^m\rightarrow s[m]$ is a representation of the affine Kac-Moody algebra $\hat{S}$ of level $k$ where 
$\langle s,s'\rangle=k\langle\langle s,s'\rangle\rangle$ for $s,s'\in S$. Since $\langle h,h\rangle \mathfrak{e}=h[1]h=2\mathfrak{e}$ and $\langle\langle h,h\rangle\rangle=2$, we then have that $k=1$. Moreover, $\overline{V_B}$ is a module of $M_{S}(1,0)$.

\vspace{0.2cm} 

\noindent Let $U$ be the vertex sub-algebra of $\overline{V_B}$ that is generated by $S$. This vertex algebra $U$ is a highest weight module for $\hat{S}$. In fact, $U$ is a quotient of the generalized Verma module $M_{S}(1,0)$. Notice that $U$ is integrable if and only if $U\cong L(1,0)$. By Theorem 10.7 in \cite{K}, this is equivalent to the condition $(e[-1])^l{\bf 1}=0$ for some $l\geq 0$. Since $(e[-1])^2{\bf 1}=0$, we can conclude immediately that $U$ is integrable. Indeed, $U$ is isomorphic to $L(1,0)$. Moreover, by Proposition \ref{nil}, and Theorem 13.16 of \cite{DoL}, $\overline{V_B}$ is integrable as $\hat{S}$-module.

\vspace{0.2cm}

\noindent By Proposition \ref{Jmodule2}, we have $f[-1]f=0$. 

\begin{lem}\label{killing} 

\ \ 

\noindent (i) $((f+h-e)[-1])^2{\bf 1}=0$. 

\noindent (ii) If $(W,Y_W)$ is a $\overline{V_B}$-module then $Y_{W}(e,z)^{2}=Y_{W}(f,z)^2=Y_{W}(f+h-e,z)^2=0$. 

\noindent (iii) In particular, we have $Y_{\overline{V_B}}(e,z)^{2}=Y_{\overline{V_B}}(f,z)^2=Y_{\overline{V_B}}(f+h-e,z)^2=0$ on $\overline{V_B}$. 
\end{lem}
\begin{proof} Since 
\begin{eqnarray*}
&&0=e[0](f[-1])^2{\bf 1}=f[-1]h[-1]{\bf 1}+h[-1]f[-1]{\bf 1},\\
&&0=f[0](e[-1])^2{\bf 1}=-(e[-1]h[-1]{\bf 1}+h[-1]e[-1]{\bf 1})\text{ and }\\
&&0=(f[0])^2(e[-1])^2{\bf 1}=-(2e[-1]f[-1]{\bf 1}-2(h[-1])^2{\bf 1}+2f[-1]e[-1]{\bf 1})
\end{eqnarray*}
we can conclude that $((f+h-e)[-1])^2{\bf 1}=0$. This proves $(i)$.

\vspace{0.2cm}

\noindent By Proposition \ref{nil}, we can conclude that if $(W,Y_W)$ is a $\overline{V_B}$-module then $Y_{W}(e,z)^{2}=Y_{W}(f,z)^2=Y_{W}(f+h-e,z)^2=0$. In particular, we have $Y_{\overline{V_B}}(e,z)^{2}=Y_{\overline{V_B}}(f,z)^2=Y_{\overline{V_B}}(f+h-e,z)^2=0$ on $\overline{V_B}$. We obtain statments $(ii)$ and $(iii)$ as desired.
\end{proof}

\begin{lem} $\overline{V_B}$ satisfies the $C_2$-condition.
\end{lem}
\begin{proof} Clearly, $$\overline{V_B}/C_2(\overline{V_B})=Span\{a+C_2(\overline{V_B}), b+C_2(\overline{V_B}), b^1[-1]....b^k[-1]{\bf 1}+C_2(\overline{V_B})~|~a\in A, b,b^i\in S, k\geq 2\}.$$ Now, we follow the proof of Proposition 12.6 in \cite{DLM}. Since $\{e,f,h\}$ forms a basis of $S$, this implies that $\{e,f, f+h-e\}$ forms a basis for $S$ as well. Observe that for $u,v\in\{e,f,f+h-e\}, w\in \overline{V_B}$, $u[-1]v[-1]w=v[-1]u[-1]w+(u_0v)[-2]w$. Since $$e[-1]e=f[-1]f=(f+h-e)[-1](f+h-e)=0,$$ we can conclude that 
\begin{eqnarray*}
\overline{V_B}/C_2(\overline{V_B})&&=Span\{a+C_2(\overline{V_B}), b+C_2(\overline{V_B}),u[-1]v+C_2(\overline{V_B}), \\
&&\ \ \ \ \ \ \ \ \ u[-1]v[-1]w+C_2(\overline{V_B})~|~a\in A, b, u,v,w\in\{e,f,f+h-e\}\},
\end{eqnarray*} and $\overline{V_B}$ is $C_2$-cofinite.
\end{proof}
\noindent This completes the proof of statement $(ii)$ of Theorem \ref{main2}.

\vspace{0.2cm}

\noindent Next, we will study $\mathbb{N}$-graded $\overline{V_B}$-modules. Observe that $A\oplus B$ generates $\overline{V_B}$ as a vertex algebra. Consequently, if $W$ is a $\overline{V_B}$-module, then $W$ is a restricted $\mathcal{L}$-module with $u(n)$ acting as $u_n$ for $u\in A\oplus B$, $n\in\mathbb{Z}$. Moreover, the set of $\overline{V_B}$-submodules is the set of $\mathcal{L}$-submodules. 

\begin{prop}\cite{LiY} Let $(W,Y_W)$ be a $V_{\mathcal{L}}$-module. Assume that for any $a,a'\in A, b\in B$,
\begin{eqnarray*}
&&Y_W(\mathfrak{e},z)u=u,\\
&&Y_W(a(-1)a',z)u=Y_W(a*a',z)u,\\ 
&&Y_W(a(-1)b,z)u=Y_W(a\cdot b,z)u,
\end{eqnarray*}
for all $u\in U$ where $U$ is a generating subspace of $W$ as a $V_{\mathcal{L}}$-module, then $W$ is naturally a $\overline{V_B}$-module. 
\end{prop}

\begin{lem} Let $(Q,Y_Q)$ be a $V_B$-module such that $Y(e(-1)e,z)u=0$ for all $u\in F$ where $F$ is a generating subspace of $Q$ as a $V_B$-module. Then $Q$ is a $\overline{V_B}$-module. 
\end{lem}
\begin{proof} By using Proposition \ref{annihilateVW}, one can obtain the above statement very easily. 
\end{proof}

\begin{lem}\label{Wmodule} Let $W=\oplus_{n=0}^{\infty}W_{(n)}$ be a $\mathbb{N}$-graded $\overline{V_B}$-module with $W_{(0)}\neq \{0\}$. Then 

\vspace{0.2cm}

\noindent (i) $W_{(0)}$ is an $A$-module with $a\cdot w=a_{-1}w$ for $a\in A$, $w\in W_{(0)}$, and $W_{(0)}$ is a module for the Lie algebra $B/A\partial (A)(\cong sl_2)$ with $b\cdot w=b_0w$ for $b\in B$, $w\in W_{(0)}$. Furthermore, $W_{(0)}$ equipped with these module structures is a module for the Lie $A$-algebroid $B/A\partial (A)$

\vspace{0.2cm} 

\noindent (ii) Moreover, $e_0(e_0w)=0$, $e_{-1}(e_{-1}w)=0$, $f_0(f_0 w)=0$ and $f_{-1}(f_{-1}w)=0$ for all $w\in W_{(0)}$. 

\vspace{0.2cm} 

\noindent (iii) If $W$ is simple then $W_{(0)}$ is an irreducible module for Lie $A$-algebroid $B/A\partial(A)$ that has dimension either 1 or 2. Moreover, for $j\in\{1,...,l\}$, $i\in\{0,1\}$, $a_{j,i}$ acts trivially on $W_{(0)}$. 
\end{lem}
\begin{proof} By following the proof of Proposition 4.8 of \cite{LiY}, one can shows that statements (i) and the statement ``if $W$ is simple then $W_{(0)}$ is a simple $B/A\partial(A)$-module'' hold.  

\vspace{0.2cm} 

\noindent Now, we will prove that $e_0e_0w=f_0f_0w=0$ for all $w\in W_{(0)}$. By Proposition \ref{killing}, we have $Y_W(e,z)^2=0$ and $Y_W(f,z)^2=0$ on $W$. Since $e_nu=0$ for all $n\geq 1$, $u\in W_{(0)}$, these imply that $Y(e,z)^2u=\sum_{n\geq 0}\sum_{m\geq 0}e_{-m}e_{-n}uz^{m+n-2}$. Moreover, the coefficient of $z^{-2}$ is $e_0e_0u=0$ and the constant term is $e_{-1}e_{-1}u=0.$ Similarly, using the fact that $Y_W(f,z)^2=0$ and $f_nu=0$ for all $n\geq 1$, $u\in W_{(0)}$, one can show that $$f_0f_0u=0\text{ and }f_{-1}f_{-1}u=0.$$ We obtain statement (ii) as desired. 

\vspace{0.2cm}

\noindent Next, we prove statement (iii). We only need to show that for an irreducible $\overline{V_B}$-module $W$, $W_{(0)}$ is either one dimensional or two dimensional. By statement (ii), $W_{(0)}$ has either one dimensional or two dimensional. Recall that $\{\mathfrak{e}, a_{j,i}~|~j\in\{1,...,l\}, i\in\{0,1\}\}$ is a basis of $A$, $B/A\partial(A)\cong sl_2$, and  
\begin{equation}\label{relvaw} v_0(a_{-1}w)-a_{-1}(v_0w)=(v_0a)_{-1}w\text{ for all }a\in A, v\in B/A\partial A, w\in W_{(0)},
\end{equation} and 
$e_0(a_{j,1})=a_{j,0}$, $f_0(a_{j,0})=a_{j,1}$.

\vspace{0.2cm}

\noindent If $dim~W_{(0)}=1$ then $W_{(0)}$ is a trivial module of $sl_2$. For simplicity, we set $W=\mathbb{C}w_0$. By equation (\ref{relvaw}) , we have 
\begin{eqnarray*}
&&0=e_0((a_{j,1})_{-1}w_0)-(a_{j,1})_{-1}e_0w_0=(e_0(a_{j,1}))_{-1}w_0=(a_{j,0})_{-1}w_0,\text{ and }\\
&&0=f_0((a_{j,0})_{-1}w_0)-(a_{j,0})_{-1}f_0w_0=(f_0(a_{j,0}))_{-1}w_0=(a_{j,1})_{-1}w_0.
\end{eqnarray*}

\vspace{0.2cm}

\noindent We now assume that $dim~ W_{(0)}=2$. Then $W_{(0)}=Span\{w_0,w_1\}$ where $w_0$ is the highest weight vector of $W_{(0)}$ of weight 1 and $w_1=f_0w_0$. Recall that for $a\in A$, $b\in B/A\partial (A)$, $a_{-1}(b_0w)=(a_{-1}b)_0w$ for all $w\in W_{(0)}$. Hence, we have
\begin{eqnarray*}
&&(a_{j,0})_{-1}w_0=(a_{j,0})_{-1}(e_0w_1)=((a_{j,0})_{-1}e)_0w_1=0,\\
&&(a_{j,0})_{-1}w_1=(a_{j,0})_{-1}(f_0w_0)=((a_{j,0})_{-1}f)_0w_0=(\partial(a_{j,1}))_0w_0=0,\\
&&(a_{j,1})_{-1}w_0=(a_{j,1})_{-1}(e_0w_1)=((a_{j,1})_{-1}e)_0w_1=(\partial(a_{j,0}))_0w_1=0,\\
&&(a_{j,1})_{-1}w_1=(a_{j,1})_{-1}(f_0w_0)=((a_{j,1})_{-1}f)_0w_0=0.
\end{eqnarray*}
This completes the proof of statement (iii).
\end{proof}

\begin{lem}\label{moduleVBS} Let $W$ be a $\mathbb{N}$-graded $\overline{V_B}$-module with $W_{(0)}\neq \{0\}$. If $W$ is an irreducible module for the vertex algebra $L(1,0)$, then $W$ is an irreducible $\overline{V_B}$-module.
\end{lem}
\begin{proof}  Let $W$ be a $\mathbb{N}$-graded $\overline{V_B}$-module with $W_{(0)}\neq \{0\}$. Assume that $W$ is an irreducible module for the vertex algebra $L(1,0)$. Let $G\neq\{0\}$ be $\overline{V_B}$-submodule of $W$. Hence, $G$ is a $U$-submodule of $W$. Since $U$ is isomorphic to $L(1,0)$ as vertex algebra, we can conclude that $G=W$. Consequently, $W$ is an irreducible $\overline{V_B}$-module.
\end{proof}

\begin{lem} Let $L=\mathbb{Z}\alpha$ be a positive definite even lattice of rank one equipped with a $\mathbb{Q}$-valued $\mathbb{Z}$-bilinear form $(\cdot,\cdot)$ such that $(\alpha,\alpha)=2$. Then the vertex operator algebra $V_L=\oplus_{n=0}^{\infty}(V_L)_{(n)}$ is an irreducible $\mathbb{N}$-graded $\overline{V_B}$-module such that $\dim (V_L)_{(0)}=1$ and $V_{L+\frac{1}{2}\alpha}=\oplus_{n=0}^{\infty}(V_{L+\frac{1}{2}\alpha})_{(n)}$ is an irreducible $\mathbb{N}$-graded $\overline{V_B}$-module such that $\dim (V_{L+\frac{1}{2}\alpha})_{(0)}=2$.
\end{lem}
\begin{proof} Recall that $(V_L, Y)$ and $(V_{L+\frac{1}{2}\alpha},Y)$ are irreducible modules of $L(1,0)$ (see Appendices). By Lemma \ref{moduleVBS}, to show that $V_L$ and $V_{L+\frac{1}{2}\alpha}$ are irreducible $\overline{V_B}$-modules, we only need to show that they are actually $\overline{V_B}$-modules. Let $L^{\circ}$ be the dual lattice of $L$. To prove that $V_L$ and $V_{L+\frac{1}{2}\alpha}$ are irreducible $\overline{V_B}$-modules, we only need to show that $V_{L^{\circ}}=V_L\oplus V_{L+\frac{1}{2}\alpha}$ is a $\overline{V_B}$-module. 

\vspace{0.3cm}

\noindent Recall that for $\gamma\in \mathbb{H}$, $\beta\in L$,
\begin{eqnarray*}
&&Y(\gamma,z)=Y(\gamma(-1){\bf 1},z)=\sum_{n\in\mathbb{Z}}\gamma(n)z^{-n-1},\\
&&Y(e^{\beta},z)=\sum_{n\in\mathbb{Z}}(e^{\beta})_nz^{-n-1}=\exp\left(\sum_{m=1}^{\infty}\beta(-m)\frac{z^m}{m}\right)\exp\left(-\sum_{m=1}^{\infty}\beta(m)\frac{z^{-m}}{m}\right)e^{\beta}z^{\beta}.
\end{eqnarray*} 
For $n\in\mathbb{Z}$, we assume that $h\otimes t^n$ acts as $\alpha(n)$, $e\otimes t^n$ acts as $(e^{\alpha})_n$, $f\otimes t^n$ acts as $(e^{-\alpha})_n$, $a_{j,i}\otimes t^n$ acts as zero, $\partial(a_{j,i})\otimes t^n$ acts as zero and $\mathfrak{e}\otimes t^n$ acts as $\delta_{n,-1}1$. First we claim that $(V_{L^{\circ}}, Y_{V_{L^{\circ}}})$ is a $V_{\mathcal{L}}$-module.
Let $a=\lambda_{\mathfrak{e}}\mathfrak{e}+\sum_{j=1}^l\sum_{i=0}^1\lambda_{j,i}a_{j,i}$. Notice that 
\begin{equation}\label{conidentity}
Y_{V_{L^{\circ}}}(a,z)=\lambda_{\mathfrak{e}}Y_{V_{L^{\circ}}}(\mathfrak{e},z)=\lambda_{\mathfrak{e}}Id_{V_{L^{\circ}}}.
\end{equation} 
Here, $Id_{V_{L^{\circ}}}$ is the identity map on $V_{L^{\circ}}$.
The following are commutator relations among $\alpha$, $e^{\pm\alpha}$ on $V_{L^{\circ}}$:
\begin{eqnarray*}
[Y(\alpha,z_1),Y(\alpha,z_2)]&&=-2\frac{\partial}{\partial z_1}z_2^{-1}\delta\left(\frac{z_1}{z_2}\right)=2\frac{\partial}{\partial z_2}z_2^{-1}\delta\left(\frac{z_1}{z_2}\right),\\
{[Y(\alpha,z_1), Y(e^{\pm\alpha},z_2)]}&&=\pm 2 z_2^{-1}\delta\left(\frac{z_1}{z_2}\right)Y(e^{\pm\alpha},z_2)\\
{[Y(e^{\alpha},z_1), Y(e^{-\alpha},z_2)]}&&=z_2^{-1}\delta\left(\frac{z_1}{z_2}\right)Y(\alpha,z_2)-\frac{\partial}{\partial z_1}z_2^{-1}\delta\left(\frac{z_1}{z_2}\right),\\
&&=z_2^{-1}\delta\left(\frac{z_1}{z_2}\right)Y(\alpha,z_2)+\frac{\partial}{\partial z_2}z_2^{-1}\delta\left(\frac{z_1}{z_2}\right),\\
{[Y(e^{\alpha},z_1),Y(e^{\alpha},z_2)]}&&={[Y(e^{-\alpha},z_1),Y(e^{-\alpha},z_2)]}=0.\\
\end{eqnarray*} By comparing these commutator relations with commutator relations (\ref{vaaa'})-(\ref{vabb'}), we can conclude that $V_{L^{\circ}}$ is a $\mathcal{L}$-module. Hence, $V_{L^{\circ}}$ is a $V_{\mathcal{L}}$-module.  

\vspace{0.2cm}

\noindent Let $a=\lambda_{\mathfrak{e}}\mathfrak{e}+\sum_{j=1}^l\sum_{i=0}^1\lambda_{j,i}a_{j,i}$, $a'=\lambda'_{\mathfrak{e}}\mathfrak{e}+\sum_{j=1}^l\sum_{i=0}^1\lambda'_{j,i}a_{j,i}\in A$. Since $$Y_{V_{L^{\circ}}}(a*a',z_2)=\lambda_{\mathfrak{e}}\lambda'_{\mathfrak{e}}Id_{V_{L^{\circ}}}$$ and 
\begin{eqnarray*}
Y_{V_{L^{\circ}}}(a(-1)a',z_2)&&=\Res_{z_0}\{z_0^{-1}\Res_{z_1}\{z_0^{-1}\delta\left(\frac{z_1-z_2}{z_0}\right)Y_{V_{L^{\circ}}}(a,z_1)Y_{V_{L^{\circ}}}(a',z_2)\\
&&\ \ \ \ \ \ \ -z_0^{-1}\delta\left(\frac{z_2-z_1}{-z_0}\right)Y_{V_{L^{\circ}}}(a',z_2)Y_{V_{L^{\circ}}}(a,z_1)\} \}\\
&&=\Res_{z_0}z_0^{-1}\Res_{z_1}z_2^{-1}\delta\left(\frac{z_1-z_0}{z_2}\right)\lambda_{\mathfrak{e}}\lambda'_{\mathfrak{e}}Id_{V_{L^{\circ}}}\\
&&=\lambda_{\mathfrak{e}}\lambda'_{\mathfrak{e}}Id_{V_{L^{\circ}}},
\end{eqnarray*} 
we can conclude that 
\begin{equation}\label{conaa'}
Y_{V_{L^{\circ}}}(a(-1)a',z_2)= Y_{V_{L^{\circ}}}(a*a',z_2).
\end{equation} 
Let $b=\rho_e e+\rho_f f+\rho_h h+\sum_{j=1}^l\sum_{i=0}^1\rho_{j,i} \partial(a_{j,i})$. Notice that $$Y_{V_{L^{\circ}}}(b,z)=\rho_eY_{V_{L^{\circ}}}(e,z)+\rho_f Y_{V_{L^{\circ}}}(f,z)+\rho_h Y_{V_{L^{\circ}}}(h,z).$$ Since 
\begin{eqnarray*}
a\cdot b&&=(\lambda_{\mathfrak{e}}\mathfrak{e}+\sum_{j=1}^l\sum_{i=0}^1\lambda_{j,i}a_{j,i})\cdot ( \rho_e e+\rho_f f+\rho_h h+\sum_{j=1}^l\sum_{i=0}^1\rho_{j,i} \partial(a_{j,i} )\\
&&=\lambda_{\mathfrak{e}}b+\tau\text{ where }\tau\in \partial(A),
\end{eqnarray*} we then have that 
$Y_{V_{L^{\circ}}}(a\cdot b,z)=Y_{V_{L^{\circ}}}(\lambda_{\mathfrak{e}}b,z)=\lambda_{\mathfrak{e}}Y_{V_{L^{\circ}}}(\rho_e e+\rho_f f+\rho_h h,z)$. Since
\begin{eqnarray*}
Y_{V_{L^{\circ}}}(a(-1)b,z_2)&&=\Res_{z_0}\{z_0^{-1}\Res_{z_1}\{z_0^{-1}\delta\left(\frac{z_1-z_2}{z_0}\right)Y_{V_{L^{\circ}}}(a,z_1)Y_{V_{L^{\circ}}}(b,z_2)\\
&&\ \ \ \ \ \ \ -z_0^{-1}\delta\left(\frac{z_2-z_1}{-z_0}\right)Y_{V_{L^{\circ}}}(b,z_2)Y_{V_{L^{\circ}}}(a,z_1)\} \}\\
&&=\Res_{z_0}z_0^{-1}\Res_{z_1}z_2^{-1}\delta\left(\frac{z_1-z_0}{z_2}\right)\lambda_{\mathfrak{e}}Y_{V_{L^{\circ}}}(b,z_2)\\
&&=\lambda_{\mathfrak{e}}Y_{V_{L^{\circ}}}(b,z_2),
\end{eqnarray*}
this implies that 
\begin{equation}\label{conab}
Y_{V_{L^{\circ}}}(a(-1)b,z_2)= Y_{V_{L^{\circ}}}(a\cdot b,z).
\end{equation} By (\ref{conidentity}), (\ref{conaa'}), (\ref{conab}), we can conclude that $V_{L^{\circ}}$ is a $V_B$-module.  

\vspace{0.2cm}

\noindent Observe that 
\begin{eqnarray*}
Y_{V_{L^{\circ}}}(e(-1)e,z_2)&&=\Res_{z_0}\{z_0^{-1}\Res_{z_1}\{z_0^{-1}\delta\left(\frac{z_1-z_2}{z_0}\right)Y_{V_{L^{\circ}}}(e,z_1)Y_{V_{L^{\circ}}}(e,z_2)\\
&&\ \ \ \ \ \ \ -z_0^{-1}\delta\left(\frac{z_2-z_1}{-z_0}\right)Y_{V_{L^{\circ}}}(e,z_2)Y_{V_{L^{\circ}}}(e,z_1)\} \}\\
&&=\Res_{z_0}\{z_0^{-1}\Res_{z_1}\{z_0^{-1}\delta\left(\frac{z_1-z_2}{z_0}\right)Y(e^{\alpha},z_1)Y(e^{\alpha},z_2)\\
&&\ \ \ \ \ \ \ -z_0^{-1}\delta\left(\frac{z_2-z_1}{-z_0}\right)Y(e^{\alpha},z_2)Y(e^{\alpha},z_1)\} \}\\
&&=Y(e^{\alpha}_{-1}e^{\alpha},z_2)\\
&&=0.
\end{eqnarray*} Hence $V_{L^{\circ}}$ is a $\overline{V_B}$-module. This completes the proof of this Lemma.
\end{proof}

\begin{lem} Let $W$ be an irreducible $\mathbb{N}$-graded $\overline{V_B}$-module. Then $W$ is either isomorphic to $V_L$ or $V_{L+\frac{1}{2}\alpha}$. 
\end{lem}
\begin{proof} Let $W=\oplus_{n=0}^{\infty}W_{(n)}$ be an irreducible $\mathbb{N}$-graded $\overline{V_B}$-module. By Lemma \ref{Wmodule}, $W_{(0)}$ is an irreducible $sl_2$-module and the dimension of $W_{(0)}$ is either 1 or 2. Since $W$ is a $L(1,0)$-module, we can conclude that $W=\oplus_{i=1}^t U^i$ is a direct sum of irreducible $L(1,0)$-modules $U^i$ where $U^i=\oplus_{n=0}^{\infty} U^i_{(n)}$ is either isomorphic to $V_L$ or $V_{L+\frac{1}{2}\alpha}$. Here, $U^i_{(n)}=W_{(n)}\cap U^i$. 

\vspace{0.2cm} 

\noindent If the dimension of $W_{(0)}$ is 1 then $W$ is isomorphic to $V_L$. Now, we assume that the dimension of $W_{(0)}$ is 2. Then $W$ is either isomorhic to $V_L\oplus V_L$ or $V_{L+\frac{1}{2}\alpha}$. If $W$ is isomorphic to $V_L\oplus V_L$ then $W_{(0)}$ is isomorphic to a sum of two trivial $sl_2$-modules which is impossible. Hence, $W$ is isomorphic to $V_{L+\frac{1}{2}\alpha}$.
\end{proof} 
\noindent This completes the proof of statement (iii) of Theorem \ref{main2}.

\section{Appendices}
\subsection{Vertex operator algebra associated with a rank one even lattice $\mathbb{Z}\alpha$ such that $(\alpha,\alpha)=2$}

 \ \

\vspace{0.2cm}

\noindent Let $L=\mathbb{Z}\alpha$ be a positive definite even lattice of rank one, i.e., a free abelian group equipped with a $\mathbb{Q}$-valued $\mathbb{Z}$-bilinear form $(\cdot, \cdot)$ such that $(\alpha,\alpha)=2$. We set $$\mathbb{H}=\mathbb{C}\otimes_{\mathbb{Z}}L$$ and extend $(\cdot,\cdot)$ to a $\mathbb{C}$-bilinear form on $\mathbb{H}$. Let 
$$\hat{\mathbb{H}}=\mathbb{H}\otimes \mathbb{C}[t,t^{-1}]\oplus \mathbb{C}K$$ 
be the affine Lie algebra associated to the abelian Lie algebra $\mathbb{H}$ so that 
$$[\alpha(m),\alpha(n)]=2m\delta_{m+n,0}K\text{ and }[K,\hat{\mathbb{H}}]=0$$ 
for any $m,n\in\mathbb{Z}$, where $\alpha(m)=\alpha\otimes t^m$. Then $\hat{\mathbb{H}}^{\geq 0}=\mathbb{H}\otimes \mathbb{C}[t]\oplus \mathbb{C}K$ is a commutative subalgebra. For any $\lambda\in\mathbb{H}$, we define a one-dimensional $\hat{\mathbb{H}}^{\geq 0}$-module $\mathbb{C}e^{\lambda}$ such that $\alpha(m)\cdot e^{\lambda}=(\lambda,\alpha)\delta_{m,0}e^{\lambda}$, and $K\cdot e^{\lambda}=e^{\lambda}$ for $m\geq 0$. We denote by $$M(1,\lambda)=U(\hat{\mathbb{H}})\otimes_{U(\hat{\mathbb{H}}^{\geq 0})}\mathbb{C}e^{\lambda}$$ the $\hat{\mathbb{H}}$-module induced from $\hat{\mathbb{H}}^{\geq 0}$-module $\mathbb{C}e^{\lambda}$. We set $M(1)=M(1,0)$. Then there exists a linear map $Y:M(1)\rightarrow\End M(1)[[z,z^{-1}]]$ such that $(M(1), Y,{\bf 1},\omega)$ carries a simple vertex operator algebra structure and $M(1,\lambda)$ becomes an irreducible $M(1)$-module for any $\lambda\in\mathbb{H}$ (\cite{FLM2}). The vacuum vector and the Virasoro element  are given by ${\bf 1}=e^0$ and $\omega=\frac{1}{4}\alpha(-1)^2{\bf 1}$, respectively.

\vspace{0.2cm}

\noindent Let $\mathbb{C}[L]$ be the group algebra of $L$ with a basis $e^{\beta}$ for $\beta\in L$ and multiplication $e^{\beta}e^{\gamma}=e^{\beta+\gamma}$ ($\beta,\gamma\in L$). The lattice vertex operator algebra associated to $L$ is given by $$V_L=M(1)\otimes \mathbb{C}[L].$$ The dual lattice $L^{\circ}$ of $L$ is the set $\{\lambda\in\mathbb{H}~|~(\alpha,\lambda)\in\mathbb{Z}\}=\frac{1}{2}L$. Note that $L^{\circ}=L\cup (L+\frac{1}{2}\alpha)$ is the coset decomposition of $L^{\circ}$ with respect to $L$. Also, we set $\mathbb{C}[L+\frac{1}{2}\alpha]=\oplus_{\beta\in L}\mathbb{C}e^{\beta+\frac{1}{2}\alpha}$. Then $\mathbb{C}[L+\frac{1}{2}\alpha]$ is a $L$-submodule of $\mathbb{C}[L^{\circ}]$. We set $V_{L+\frac{1}{2}\alpha}=M(1)\otimes \mathbb{C}[L+\frac{1}{2}\alpha]$. It was shown in \cite{B1}, \cite{D}, \cite{FLM2}, \cite{Gu} that $V_L$ is a rational vertex operator algebra. Furthermore, $V_L$ and $V_{L+\frac{1}{2}\alpha}$ are the only irreducible modules for $V_L$ under the following action: for $\beta\in \mathbb{H}$ write $\beta(z):=\sum_{n\in\mathbb{Z}}\beta(n)z^{-n-1}$, $z^{\beta}:e^{\gamma}\mapsto z^{(\beta, \gamma)}e^{\gamma}$ and set $$Y(e^{\beta},z):=\exp\left(\sum_{m=1}^{\infty}\beta(-m)\frac{z^m}{m}\right)\exp\left(-\sum_{m=1}^{\infty}\beta(m)\frac{z^{-m}}{m}\right)e^{\beta}z^{\beta},$$ 
and for $v=\alpha_1(-n_1)...\alpha_t(-n_t)\otimes e^{\beta}\in V_L$ $(n_i\geq 1)$ set
$$Y(v,z):=\left(\frac{1}{(n_1-1)!}\left(\frac{d}{dz}\right)^{n_1-1}\alpha_1(z)\right)...\left(\frac{1}{(n_t-1)!}\left(\frac{d}{dz}\right)^{n_t-1}\alpha_t(z)\right)Y(e^{\alpha},z):,$$ with the usual normal ordering conventions.

\subsection{Vertex operator algebras associated with highest weight representations of affine Lie algebras}

\ \ 

\vspace{0.2cm}

\noindent Let $\mathfrak{g}$ be a simple Lie algebra over $\mathbb{C}$, $\mathfrak{h}$ its Cartan subalgebra and $\Delta$ the corresponding  root system. We fix a set of positive root $\Delta_+$ and a nondegnerate symmetric invariant bilinear form $\langle\cdot,\cdot \rangle$ of $\mathfrak{g}$ such that the square length of a long root is 2. Let $$\hat{\mathfrak{g}}=\mathfrak{g}\otimes \mathbb{C}[t,t^{-1}]\oplus \mathbb{C}c$$ be the affine Lie algebra with Lie bracket defined by $$[u\otimes t^m,v\otimes^n]=[u,v]\otimes t^{m+n}+m\langle u,v\rangle \delta_{m+n,0}c.$$ Here, $u,v\in\mathfrak{g}$ and $m,n\in\mathbb{Z}$ and $c$ is a central element. We will write $u(n)$ for $u\otimes t^n$. 

\vspace{0.2cm}

\noindent Let $l$ be a complex number such that $l\neq-\check{h}$ where $\check{h}$ is the dual Coxeter number of $\mathfrak{g}$. Let $\mathbb{C}_l$ be the one-dimensional $(\mathfrak{g}\otimes \mathbb{C}[t]+\mathbb{C}c$)-module on which $c$ acts as scalar $l$ and $\mathfrak{g}\otimes \mathbb{C}[t]$ acts a zero. Form the generalized Verma $\hat{\mathfrak{g}}$-module $$M_{\mathfrak{g}}(l,0)=U(\hat{\mathfrak{g}})\otimes_{U(\mathfrak{g}\otimes \mathbb{C}[t]+\mathbb{C}c})\mathbb{C}_l.$$ We define 
\begin{eqnarray*}
&&Y(u(-1)\otimes 1,z)=\sum_{n\in\mathbb{Z}}u(n)z^{-n-1},\\
&&Y(u(-m-1)\otimes 1,z)=\frac{1}{m!}\frac{d^m}{dz^m}Y(u(-1)\otimes 1,z).
\end{eqnarray*}
In general, if $Y(v,z)$ has been defined, we defined $Y(u(-n)v,z)$ for $u\in\mathfrak{g}$ and $n>0$ as 
$$Y(u(-n)v,z)=\Res_{z_1}\{(z_1-z_2)^{-n}Y(u,z_1)Y(v,z_2)-(-z_2+z_1)^{-n}Y(v,z_2)Y(u,z_1)\}.$$
We then get a linear map $Y$ from $M_{\mathfrak{g}}(l,0)$ to $(\End M_{\mathfrak{g}}(l,0)[[z,z^{-1}]]$. Set ${\bf 1}=1\otimes 1$ and $\omega=\frac{1}{2(l+\check{h})}\sum_i v_i(-1)^2\otimes 1$ where $\{v_i\}$ is an orthonormal basis of $\mathfrak{g}$ with respect to $\langle~,~\rangle$.

\begin{prop}\label{Jmodule1}\cite{FS,FZ, LLi, Li, MeP} 

\ \  

\noindent (i) $(M_{\mathfrak{g}}(l,0), Y, {\bf 1},\omega)$ is a vertex operator algebra. The category of weak $M_{\mathfrak{g}}(l,0)$-modules in the sense that all axioms defining the notion of module except those involving grading hold is canonically equivalent to the category of restricted $\hat{\mathfrak{g}}$-modules of level $l$ in the sense that for every vector $w$ of the module, $(\mathfrak{g}\otimes t^n\mathbb{C}[t])w=0$ for $n$ sufficiently large. 

\vspace{0.2cm}

\noindent (ii) Let $J_{l,0}$ be the maximal proper submodule of $M_{\mathfrak{g}}(l,0)$. Let $e_{\theta}$ be an element in the root space $\mathfrak{g}_{\theta}$ of the maximal root $\theta$. The space $J_{l,0}$ is generated by the vector $e_{\theta}(-1)^{l+1}{\bf 1}$, i.e., every element in $J_{l,0}$ can be written as a linear combination of elements of type 
$$u_1(-n_1)....u_m(-n_t)e_{\theta}(-1)^{l+1}{\bf 1}.$$ Note that $\bf{1}$ and $\omega$ are not members of $J_{l,0}$. 
\end{prop} 

\noindent Given a $\mathfrak{g}$-module $W$, we will write $u_W(n)$ for the operator on $W$ corresponding to $u\otimes t^n$ for $u\in\mathfrak{g}$ and $n\in\mathbb{Z}$ and we set $u_W(z)=\sum_{n\in\mathbb{Z}}u_W(n)z^{-n-1}\in(\End W)[[z,z^{-1}]]$. $W$ is a restricted module of level $l$ if $c$ acts as $l$ and for every $u\in\mathfrak{g}$, $w\in W$, $u(n)w=0$ for $n$ sufficiently large. 

\vspace{0.2cm}

\noindent We set $L(l,0)=M_{\mathfrak{g}}(l,0)/J_{l,0}$. 

\begin{prop}\label{Jmodule2}\cite{FZ, LLi, DoL} 

\ \ 

\noindent (i) $L(l,0)$ is a rational vertex operator algebra.

\vspace{0.2cm}

\noindent (ii) Let $\alpha\in \Delta$ be a long root, i.e., $\langle\alpha,\alpha\rangle=2$ and let $e\in \mathfrak{g}_{\alpha}$. Then $e(-1)^{l+1}{\bf 1}\in J_{l,0}$. In particular, $e_{\theta}(-1)^{l+1}{\bf 1},f_{\theta}(-1)^{l+1}{\bf 1}\in J_{l,0}$. Moreover, $Y(e,x)^{l+1}=0$ on $L(l,0)$. For any module $W$ for $L(l,0)$ viewed as a vertex algebra, $Y_W(e,z)^{l+1}=0$. In particular, $Y_W(e_{\theta},z)^{l+1}=Y_W(f_{\theta},z)^{l+1}=0$.

\vspace{0.2cm}

\noindent (iii) If $\mathfrak{g}=sl_2$ and $l=1$, then $V_L$ and $V_{L+\frac{1}{2}\alpha}$ are the only irreducible $L(1,0)$-modules.
\end{prop}

\end{document}